\documentclass[leqno,11pt,a4paper]{amsart}
\usepackage{bw_jc}
\usetikzlibrary{cd}
\newcommand{\rz}{\R{\cdot}\Z}

\newcommand{\chom}{c^{\text{\textnormal{hom}}}}

\usepackage[justification=centering]{caption}

\begin{document}
\author[J.~Chaidez]{Julian Chaidez}
\address{Department of Mathematics \\ Princeton University \\ Princeton \\ NJ \\ 08544 \\ USA}
\email{jchaidez@princeton.edu}
\vs

\title[Newton--Okounkov bodies \& embeddings]{Newton--Okounkov bodies and symplectic embeddings into non-toric rational surfaces}
\maketitle

\begin{abstract} We develop new methods of both constructing and obstructing symplectic embeddings into non-toric rational surfaces using the theory of Newton-Okoukov bodies. Applications include sharp embedding results for concave toric domains into non-toric rational surfaces, and new cases of non-existence for infinite staircases in the non-toric setting.
\end{abstract}

\section{Introduction} \label{sec:intro}

A symplectic embedding is a smooth embedding between symplectic manifolds of the same dimension that respects symplectic forms. The study of such maps is at the heart of symplectic geometry \cite{gro_pse_85,sch_sym_18} and many basic questions about their existence or non-existence remain open.

\vspace{3pt}

Over the last decade, dramatic progress has been made in dimension four (i.e.~the first non-trivial dimension). This was initiated by Hutchings' foundational work on embedded contact homology (ECH) \cite{hut_qua_11,hut_lec_14} and more specifically the introduction of \emph{ECH capacities}
\[\cech_k(X,\omega) \in (0,\infty] \qquad\text{ of a symplectic $4$-manifold }(X,\omega) \quad\text{and}\quad k \in \N\]
Soon after their introduction, McDuff used ECH capacities to resolve Hofer's conjecture on ellipsoid embeddings \cite{mcd_hof_11}. This initiated a flurry of progress on symplectic embeddings between ellipsoids and other toric domains \cite{ccfhr_sym_14,ck_ehr_20,chmp_inf_20}. In particular, Hofer's conjecture was extended to a more general class of embeddings between toric domains by Cristofaro-Gardiner \cite{cri_sym_19}.

\vspace{3pt}

Connections between symplectic embeddings and algebraic geometry have been explored by many authors \cite{bir_fro_01,mp_sym_94}. In particular, recent work of the authors \cite{wor_ech_21,cw_ech_20,wor_tow_22,wor_alg_22} has focused on the connections between ECH capacities and positivity in algebraic geometry (c.f.~Lazarsfeld \cite{laz_pos_17}). In this story, a central role is played by an algebraic version of ECH capacities called \emph{algebraic capacities}
\[\calg_k(Y,A) \qquad\text{of a polarized projective surface }(Y,A) \quad\text{and}\quad k \in \N\]
As with any such bridge between fields, benefits flow in both directions \cite{bir_con_99,lms_gro_13,wor_alg_22,cw_ech_20}. However, nearly all of the existing sharp embedding results that are approachable through this bridge apply only in the toric setting.

\vspace{3pt}

In this paper, we develop a number of tools from algebraic positivity to study symplectic embeddings from toric domains into non-toric targets. In particular, we focus on Zariski decomposition and Newton--Okounkov bodies in algebraic positivity and use them to study symplectic embedding obstructions coming from ECH. As an application, we provide new sharp results on many symplectic embedding problems with non-toric target.

\subsection{Newton--Okounkov bodies} Given an $n$-dimensional, smooth projective variety $Y$ equipped with an effective $\R$-divisor $A$ and a $\Z^n$-valued valuation $\nu\colon\C(Y)\to\Z^n$ on the field of rational functions $\C(Y)$ there is an associated \emph{Newton--Okounkov body}
\[\Delta(Y,A,\nu) \subseteq \R^n\]
Newton--Okounkov bodies are convex bodies in $\R^n$ generalizing the perhaps more familiar notion of moment polytopes in toric geometry. In particular, there is a relationship between counts of sections of the line bundles $\mO(mA)$ and lattice points in the dilates $m\Delta(Y,A,\nu)$, and the normalized volume of $\Delta(Y,A,\nu)$ is equal to the volume of the pair $(Y,A)$ \cite{laz_pos_17,lm_con_09}. The general expectation of the theory of Newton--Okounkov bodies is that many, if not all, quantifications of the algebraic positivity of $(Y,A)$ can be accessed via the combinatorics of Newton--Okounkov bodies.

\vspace{3pt}

In dimension two every Newton--Okounkov body has an associated \emph{combinatorial weight sequence}. The weight expansion of a concave or convex region $\Omega \subseteq \R_{\geq0}^2$ with two adjacent boundary edges on the coordinate axes \cite{mcd_hof_11,cri_sym_19,ccfhr_sym_14} is a sequence of real numbers describing a decomposition of $\Omega$ into triangles.

\vspace{3pt}

In the toric setting the moment polytope and associated weight sequence play a key role in the study of obstructions to embeddings between \emph{toric domains}. A toric domain is a symplectic $4$-manifold $X_\Omega$ given as preimage of a certain type of region $\Omega\subseteq\R^2$ under the moment map $\mu\colon\C^2\to\R^2$ for the standard $2$-torus action on $\C^2$. The ECH capacities of convex and concave toric domains can be entirely computed using the moment polytope and weight expansion \cite[\S A]{cri_sym_19}.

\subsection{Symplectic embeddings via Newton--Okounkov bodies} The first main result of this paper provides machinery for constructing symplectic embeddings using Newton--Okounkov bodies by extending work of Kaveh \cite{kav_tor_19}.

\vspace{3pt}

Let $(Y,A)$ be a smooth projective variety of dimension $n$ equipped with a flag
$$Y_\bullet:Y=Y_n\supseteq Y_{n-1}\supseteq\dots\supseteq Y_1\supseteq Y_0$$
where $Y_i$ is an irreducible subvariety of $Y$ of dimension $i$. $Y_\bullet$ determines a $\Z^n$-valued valuation $\nu_{Y_\bullet}$ on $\C(Y)$ and so there is an associated Newton--Okounkov body
\[\Delta = \Delta(Y,A,Y_\bullet) := \Delta(Y,A,Y_\bullet) \subseteq \R_{\geq0}^n\]
There is an associated open toric \emph{free domain} $F_\Delta$ to $\Delta$ given by the Lagrangian product
\[
F_\Delta := T^n \times \on{int}(\Delta) \subset T^n \times \R^n \simeq T^*T^n
\]
In the language of projective geometry this is the variety $(\C^\times)^n$ equipped with a certain toric K\"ahler form. In \cite{kav_tor_19} Kaveh uses a type of 'weighted' deformation to the normal cone and gradient--Hamiltonian flow to prove the following result.

\begin{thm*}[\!{\cite[Thm.~10.5]{kav_tor_19}}] \label{thm:intro_kaveh} Let $(Y,A)$ be a polarized smooth variety with a flag $Y_\bullet$ on $Y$. Let $\Delta=\Delta(Y,A,Y_\bullet)$ be the associated Newton--Okounkov body. Suppose $X \subseteq F_\Delta$ is a compact domain, then there is a symplectic embedding
\[X \to (Y,\omega_A)\]
\end{thm*}

\noindent Drawing on work of Pabiniak \cite{pab_gro_14} Kaveh applies Thm.~\ref{thm:intro_kaveh} to give a lower bound the Gromov width of projective varieties in terms of their Newton--Okounkov bodies \cite[Cor.~12.4]{kav_tor_19} and study symplectic ball packings in such spaces \cite[Cor.~12.6]{kav_tor_19}.

\vspace{3pt}

Now suppose that the Newton--Okounkov body $\Delta$ contains a neighborhood of the origin. Such a convex body has an associated toric domain $X_\Delta \subset \C^n$ and an identification
\[F_\Delta \simeq X_\Delta \setminus \{\text{points with non-trivial stabilizer}\} \subseteq X_\Delta\]
In general (and particularly in dimension four), more is known about symplectic embeddings involving toric domains than about their corresponding free domain. With this in mind, we prove the following enhancement of Thm.~\ref{thm:intro_kaveh}. 

\vspace{3pt}

To state our generalization, recall that there is a canonical sequence of polytopes $\Delta_m(Y,A,Y_\bullet)$ associated to any triple $(Y,A,Y_\bullet)$ as above, given by the convex hull of the image of $H^0(mA)$ under the valuation $\nu_{Y_\bullet}$. The normalizations $\frac{1}{m}\Delta_m(Y,A,Y_\bullet)$ converges to $\Delta(Y,A,\nu)$ as $m \to \infty$.

\begin{prop*}[Prop. \ref{thm:Kaveh_embeddings}] \label{prop:intro} Let $(Y,A)$ be a smooth, polarized variety with a flag $Y_\bullet$ and let $X \subset \intr{X_\Omega}$ be a subset of the interior of the toric domain with moment polytope
\[\Omega = \frac{1}{m}\Delta_m(Y,A,Y_\bullet) \qquad\text{for some }m\]Then there is a symplectic embedding $X \to (Y,\omega_A)$. 
\end{prop*}

\noindent The polytopes $\frac{1}{m}\Delta_m(Y,A,Y_\bullet)$ exhaust the Newton-Okounkov body $\Delta$. Therefore, Prop. \ref{prop:intro} morally states that embeddings into the interior of the toric domain $X_\Delta$ may be transfered to $Y$ itself.  We refer to the embeddings resulting from this construction as \emph{Kaveh embeddings}. 

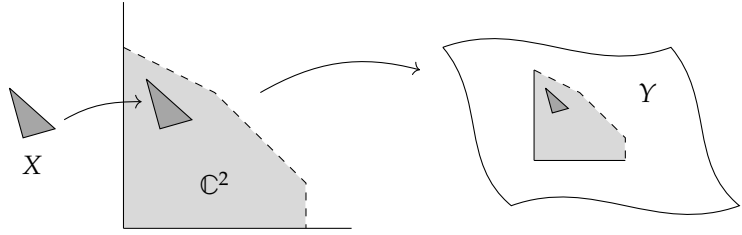
\begin{figure}[h]
    \centering
    \begin{tikzpicture}[scale=0.6]

    \draw[white,fill=gray!30] (8,0) to (8,4) to (10,3) to (12,1) to (12,0) to (8,0);
    \draw[dashed] (8,4) to (10,3) to (12,1) to (12,0);
    \draw (8,0) to (8,5);
    \draw (8,0) to (13,0);
    
    \draw[fill=gray!70] (8.5,3.3) to (9.5,2.4) to (8.8,2.2) to (8.5,3.3);
    
    \draw[fill=gray!70] (5.5,3.1) to (6.5,2.2) to (5.8,2) to (5.5,3.1);
    
    \draw (15,4) [out=310,in=140] to (16.5,0.5) [out=20,in=200] to (21.5,0.5) [out=140,in=310] to (20,4) to [out=200,in=20] (15,4);
    
    \draw[white,fill=gray!30] (17,1.5) to (17,3.5) to (18,3) to (19,2) to (19,1.5) to (17,1.5);
    \draw[dashed] (17,3.5) to (18,3) to (19,2) to (19,1.5);
    \draw (17,1.5) to (17,3.5);
    \draw (17,1.5) to (19,1.5);
    \draw[fill=gray!70] (17.25,3.1) to (17.75,2.65) to (17.4,2.55) to (17.25,3.1);
    
    \draw[->] (6.7,2.4) [out=30,in=180] to (8.4,2.8);
    \draw[->] (11,3) [out=30,in=165] to (14.5,3.5);
    
    \node (lx) at (6,1.4){\small $X$};
    \node (ly) at (19.5,3){\small $Y$};
    \node (la) at (10,1){\small $\C^2$};
    \end{tikzpicture}
    \caption{Transferring embeddings from $\intr{X_{\frac{1}{m}\Delta_m}}$ to $(Y,\omega_A)$}
    \label{fig:my_label}
\end{figure}

\noindent The proof largely follows Kaveh's original approach, by using a family of projective varieties acquired by deformation to the normal cone where the central fiber is symplectomorphic to $\C^2$ equipped with the toric symplectic form prescribed by $\Delta_m$.

\subsection{Algebraic capacities via Newton--Okounkov bodies} The second main result of this paper provides machinery for obstructing symplectic embeddings, by computing algebraic capacities combinatorially using Newton-Okounkov bodies. 

\vspace{3pt}

Let $(Y,A)$ be a \emph{pseudo-polarized rational surface}; i.e. a pair consisting of a smooth rational surface $Y$ and a big and nef (or ample) $\R$-divisor $A$ on $Y$. Then $Y$ (or a blowup of $Y$) can be presented as a tower of point blowups
\begin{equation} \tag{$\star$} \label{eqn:tower_intro}
    Y=Y_r\overset{\pi_r}{\longrightarrow} Y_{r_1}\overset{\pi_{r-1}}{\longrightarrow}\dots\overset{\pi_2}{\longrightarrow} Y_1\overset{\pi_1}{\longrightarrow}\pr^2
\end{equation}
with exceptional curve $E_i\subseteq Y_i$ of each $\pi_i$ so that there is a sequence of divisors with $A_0=\mO_{\pr^2}(c)$ and $A_i=\pi_i^*A_{i-1}-a_iE_i$ on $Y_i$ with $A_n=A$. We denote such a tower of polarized surfaces by $(\mathcal{Y},\mathcal{A})$. The \emph{algebraic weight sequence} of the tower $(\mathcal{Y},\mathcal{A})$ is then
\begin{equation} \tag{$+$} \label{eqn:wt_sequence_intro}
    \on{wt}(\mathcal{Y},\mathcal{A}):=(c;a_1,\dots,a_n)
\end{equation}
Note that there are many different presentations of $(Y,A)$ as a tower with potentially different weight sequences. We say that $(Y,A)$ \emph{admits} the weight sequence (\ref{eqn:wt_sequence_intro}) if it arises from a tower presentation of $(Y,A)$.

\vspace{3pt}

We show that, as for the combinatorial weight sequence and ECH capacities, the algebraic weight sequence determines the algebraic capacities of a polarized surface.

\begin{thm*}[Prop.~\ref{prop:bounds}] \label{thm:wt_seq_intro}
Let $(Y,A)$ be a pseudo-polarized rational surface. Assume that there is a pseudo-polarized rational surface $(Y_+,A_+)$ with the same weight sequence as $(Y,A)$ and with $-K_{Y_+}$ effective. Then
$$\calg_k(Y,A)=\calg_k(Y_+,A_+)$$
\end{thm*}

\noindent This provides a valuable computational tool for algebraic capacities on surfaces with poorly behaved nef cones; for instance, general blowups of $\pr^2$ in more than $8$ points. This is entirely analogous to, and actually implies, \cite[Thm.~A.1]{cri_sym_19} for ECH capacities.

\vspace{3pt}

We state the main consequence Thm.~\ref{thm:wt_seq_intro} has for algebraic capacities and their connection to ECH capacities for toric domains.

\begin{thm*}[Thm.~\ref{thm:alg_ech}] \label{thm:main_intro} Let $(Y,A)$ be a polarized rational surface. Suppose there is a Newton--Okounkov body $\Delta=\Delta(Y,A,\nu)$ whose weight sequence is admitted by $(Y,A)$. Then the algebraic capacities of $(Y,A)$ and the ECH capacities of $F_\Delta$ coincide.
\[\calg_k(Y,A) = \cech_k(F_\Delta)\]
If in addition $\Delta$ is \emph{$A$-generic} (see Def.~\ref{def:a_generic}) then $\Delta$ is a moment domain and so
\[\calg_k(Y,A) = \cech_k(X_\Delta)\]
\end{thm*}

\noindent Note that if $\Delta$ contains a neighborhood of the origin in $\R_{\geq0}^2$ then $\cech_k(F_\Delta) = \cech_k(X_\Delta)$. This is expressed algebraically by the (mild) condition that $\Delta$ is $A$-generic.

\vspace{3pt}

As an example consequence of Theorem \ref{thm:main_intro}, we can given algebraic formula for the Gromov width of some polarized rational surfaces. 

\begin{cor*} \label{cor:gromov_width} Let $(Y,A)$ be a polarized rational surface. Suppose there is a strongly convex Newton--Okounkov body $\Delta=\Delta(Y,A,\nu)$ whose weight sequence is admitted by $(Y,A)$. Then
\[
c_G(Y,\omega_A) = \calg_1(Y,A) = \inf_{D\in\on{Nef}(Y)_\Z}\{A\cdot D:I(D)\geq 2\}
\]
\end{cor*}

\begin{proof} The Gromov width and the first ECH capacity agree for strongly convex, free toric domains (cf. \cite{ghr_exa_22}), and so $c_G(F_\Delta) = \cech_1(F_\Delta)$. A ball that embeds into $F_\Delta$ also embeds into $(Y,\omega_A)$ by Thm. \ref{thm:intro_kaveh}. Therefore, we have
\[c_G(Y,\omega_A) \ge c_G(F_\Delta) = \cech_1(F_\Delta) = \calg_1(Y,A) \]
Finally, we apply the main theorem of \cite{cw_ech_20} to see that
\[c_G(Y,\omega_A) \le \calg_1(Y,A) \qedhere\]    
\end{proof}

\subsection{Applications} \label{subsec:applications_and_computations} Our main results yield a number of nice symplectic embedding results for non-toric projective surfaces.

\subsubsection{Computation of algebraic capacities} As a first application we are able to identify many non-toric rational surfaces where the conditions of Thm.~\ref{thm:main_intro} are met.

\vspace{3pt}

For each family of surfaces $(Y,A)$ we consider we find a Newton--Okounkov body $\Delta$ that completely calculates the algebraic capacities of $(Y,A)$. Moreover, via Thm.~\ref{prop:intro} we acquire embeddings $X_\Delta \to Y$ with complement of arbitrarily small volume. This analysis is carried out for many examples in \S\ref{sec:app}. In this introduction we will discuss a few of these computations, starting with the following example.

\begin{thm*}[Prop.~\ref{prop:example} + Prop.~\ref{prop:higher}] \label{thm:higher_intro}
Let $Y$ be a rational surface whose only curves of negative self-intersection are rational curves $C$ with $C^2=-1$. Let $A$ be an ample $\R$-divisor on $Y$ such that $(Y,A)$ admits a weight sequence $(c;a_1,\dots,a_n)$ with
$$c\geq\frac{n-3}{2}\sum_{i=1}^na_i$$
Then there is an $A$-generic Newton--Okounkov body $\Delta$ whose weight sequence is admitted by $(Y,A)$.
\end{thm*}

\noindent When $Y$ is a del Pezzo surface there are weaker inequalities depending on the Picard rank for $c,a_1,\dots,a_n$ to satisfy. All del Pezzo surfaces are covered by Thm.~\ref{thm:higher_intro}. If $n\leq 3$ then $Y$ must be a toric del Pezzo surface. In that case the inequality is vacuous and the result holds by direct calculation. 

\vspace{3pt}

A consequence of the well-known SHGH Conjecture \cite[Conj.~0.1]{def_neg_04} or \cite[Conj.~2.2.3]{chmr_var_13} gives that all blowups of $\pr^2$ in $n$ very general points also fit into the setting of Thm.~\ref{thm:higher_intro}. Recall that a set of $n$ points being ``very general'' means that it lies in the complement of a countable union of subvarieties in the configuration space of all sets of $n$ points.

\vspace{3pt}

As a second example, we formulate a general procedure for producing non-toric examples that our main results apply to. This procedure may be thought of as a ``genericization'', moving the centers of a tower of blowups to be in more general position.

\begin{thm*}[Prop.~\ref{prop:bounds} + Prop.~\ref{prop:flat}] \label{thm:compute_intro}
Let $(Y_\times,A_\times)$ be a pseudo-polarized toric surface presented as a tower as in (\ref{eqn:tower_intro}). Suppose $(Y,A)$ is a pair consisting of a rational surface $Y$ and an $\R$-divisor $A$ such that:
\begin{itemize}
    \item $(Y,A)$ can be presented as a tower with the same weight sequence as $(Y_\times,A_\times)$,
    \item the centres of the blowups for $Y$ are in more general position than those for $Y_\times$.
\end{itemize}
Then $A$ is big and nef and $\calg_k(Y,A)=\calg_k(Y_\times,A_\times)$.
\end{thm*}

\noindent One can think of `big and nef' as meaning `possibly degenerate symplectic form' and so the first part of Thm.~\ref{thm:compute_intro} states that $(Y,A)$ can indeed be regarded as an object of symplectic geometry. The second part relates the algebraic capacities of $(Y,A)$ -- which may be hard to compute as the nef cone of $Y$ could be very complicated -- to the algebraic capacities of $(Y_\times,A_\times)$, which are highly accessible by algebraic or symplectic methods reducing to lattice combinatorics  \cite{wor_ech_21,cri_sym_19,hut_qua_11}.

\subsubsection{Concave to non-toric embeddings} Our main results also provide an immediate generalization of the optimal embedding result of Cristofaro-Gardiner on symplectic embeddings of concave toric domains into convex ones to the closed, non-toric setting.

\begin{cor*} \label{cor:intro_concave_to_convex} Let $(Y,A)$ be a pseudo-polarized rational surface. Suppose that there is a Newton--Okounkov body $\Delta(Y,A,Y_\bullet)$ whose weight sequence is admitted by $(Y,A)$. Then, for any concave toric domain $X_\Omega$ the following are equivalent.
\begin{itemize}
    \item For any $r < 1$, there is a symplectic embedding $X_{r\Omega} \to (Y,\omega_A)$.
    \item The ECH capacities $\cech_k(X_\Omega)$ are bounded above by the algebraic capacities $\calg_k(Y,A)$.
\end{itemize}
\end{cor*}

\noindent This result applies to all of the examples computed in \S\ref{sec:app}; the del Pezzo surfaces and higher rank blowups in Thm.~\ref{thm:higher_intro}, and the genericizations of toric surfaces from Thm.~\ref{thm:compute_intro}). This is a rare example of a sharp result for symplectic embeddings into non-toric target manifolds.

\begin{remark} Variants of the existence result for symplectic embeddings in Cor.~\ref{cor:intro_concave_to_convex} can also be proven by adapting the proof methods of Cristofaro-Gardiner--Holm--Mandini--Pires (see in particular \cite[Rem.~3.6]{chmp_inf_20}) to the non-toric context. Our proof adheres to the general philosophy of this paper by utilizing Kaveh embeddings from Newton--Okounkov bodies to produce non-toric embeddings from toric ones.
\end{remark}

\subsubsection{Obstructing staircases} Recall that the \emph{ellipsoid embedding function} $f_X:\R_+ \to \R_+$ of a symplectic 4-manifold $X = (X,\omega)$ is the function
\[f_X(a) := \inf\{r \; : \; E(1,a) \to (X,r \cdot \omega)\}\]
A famous result of McDuff--Schlenk \cite{ms_emb_12} demonstrates the existence of certain \emph{infinite staircases} occurring in the ellipsoid embedding function of certain toric domains. An infinite staircase consists of the infinite sequence of points of nondifferentiability of $f_X$ accumulating to a finite point $a_0 \in \R_{\geq1}$.

\vspace{3pt}

Since the work of McDuff-Schlenk \cite{ms_emb_12} the characterization of spaces that possess an infinite staircase has attracted significant interest \cite{chmp_inf_20,mmw_sta_22}. As a final application we provide some initial applications of our machinery to the intricate problem of obstructing infinite staircases in non-toric rational surfaces.

\begin{prop*}[Prop.~\ref{prop:staircase}]
Let $(Y,A)$ be a polarized del Pezzo surface. The ellipsoid embedding function of $(Y,\omega_A)$ has no infinite staircase when:
\begin{enumerate}
    \item $Y$ has degree $3$ and $(Y,A)$ has weight sequence
    $$(c;1,1,1,1,1,1)\qquad\text{for $4\leq c<\frac{18+\sqrt{24}}{5}\approx 4.57$}$$
    \item $Y$ has degree $1$ and $(Y,A)$ has weight sequence
    $$(c;1,1,1,1,1,1,1,1)\qquad\text{for $6\leq c<\frac{24+\sqrt{96}}{5}\approx 6.76$}$$
\end{enumerate}
\end{prop*}

\subsection{Asymptotics via Zariski decomposition} \emph{Zariski decomposition} \cite{zar_the_62} splits a big divisor $A$ on a surface as
\[A=P+N\]
where $P$ is big and nef and $N$ is pseudo-effective. The primary feature of this decomposition is that all `positivity' information about $A$ (e.g.~section counts) is captured by $P$. Zariski decomposition will play a significant role in the proofs and computations of this paper.

\vspace{3pt}

We can also draw on Zariski decomposition to extend the algebraic Weyl law \cite[Thm.~4.2]{wor_alg_22} from polarized surfaces to \emph{weakly polarized surfaces}, i.e. a pair $(Y,A)$ where $Y$ is a $\Q$-factorial projective surface and $A$ is a big divisor. This brings in a central invariant in algebraic positivity: the \emph{volume} $\on{vol}(A)$ of a divisor $A$, which roughly tracks the growth of sections of $mA$ for large $m$.

\begin{thm*}[Prop.~\ref{prop:weyl_big}] \label{thm:weyl_intro} Suppose $(Y,A)$ is a smooth or toric weakly polarized surface. Then
$$\lim_{k\to\infty}\frac{\calg_k(Y,A)^2}{k}=2\on{vol}(A)$$
\end{thm*}

\noindent As with the algebraic Weyl law for big and nef divisors, Thm.~\ref{thm:weyl_intro} allows us to define \emph{error terms}
$$\ealg_k(Y,A):=\calg_k(Y,A)-\sqrt{2\on{vol}(A)k}$$
to capture the sub-leading asymptotics of $\calg_k(Y,A)$. In the case that $A=qA_0$ for some $\Z$-divisor $A_0$ and some $q\in\R_{>0}$ we can give a precise description of these sub-leading asymptotics, just as in the big and nef case \cite[Thm.~4.10]{wor_alg_22}.

\begin{thm*}[Cor.~\ref{cor:sublead}]
Let $(Y,A)$ be a weakly polarized surface that is either smooth or toric, and where $A$ is a real multiple of a $\Z$-divisor. Then the error term $\ealg_k(Y,A)$ is $O(1)$ and nonconvergent with explicitly computable lim inf and lim sup.
\end{thm*}

As with the big and nef case \cite[\S3.8]{wor_tow_22} it is an intriguing open question to more precisely characterise when convergence does and does not occur among the sub-leading asymptotics for algebraic capacities of weakly polarized surfaces.

\subsection{Future directions} We conclude by remarking on a few possible extensions of this work.

\begin{remark}[Abelian surfaces]
While the results of \cite{cw_ech_20} and hence of this paper do not apply to non-rational surfaces, we conclude with an example illustrating possible connections in the case of abelian surfaces (or four dimensional complex tori). The Newton--Okounkov body for a polarized abelian surface is a trapezium \cite[Ex.~6.5]{lm_con_09}, which describes a ruled surface over $\C\pr^1$ in toric algebraic geometry or a related convex toric domain from the perspective of \S\ref{sec:symp_emb}. Ball packings of abelian surfaces have been studied in two different ways using ball packings of associated ruled surfaces \cite{lms_gro_13,bir_sta_99} and it is plausible that this structure is being reflected in their Newton--Okounkov theory.
\end{remark}

\begin{remark}[Other connections] To obtain a wider class of sharp embedding results we need to construct Newton--Okounkov bodies with prescribed weight sequences. We hope that the Newton--Okounkov bodies coming from cluster varieties \cite{bcmn} or from a finer understanding of the relationship between Newton--Okounkov bodies and toric degenerations \cite{and_oko_13} will allow for these next steps. It would also be interesting to understand the effect of blowup and perturbation on the global Newton--Okounkov body \cite[Thm.~B]{lm_con_09} as a way of potentially systematising the process of computing weight sequences of Newton--Okounkov bodies that we have developed here.
\end{remark}

\begin{remark}[Limitations]
    We make the cautionary note that using Zariski decomposition to calculate Newton--Okounkov bodies is currently a radically difficult task in general, and that we were forced to use the still-open SHGH Conjecture to complete our applications to rational surfaces of higher Picard rank. We emphasise that there is no intrinsic reason why our methods should not generalise far beyond the applications presented here, indeed this seems likely as our understanding of Newton--Okunkov bodies continues to grow.
\end{remark}

\subsection*{Acknowledgements} The authors are very grateful to Alex K\"uronya for frequent helpful conversations around the subject of this paper. They are also grateful for input from Dave Anderson, Laura Escobar, Tara Holm, Tim Magee, Nicki Magill, Dusty Ross, and Morgan Weiler.

\section{Algebraic preliminaries} \label{sec:alg} In this section we give a detailed discussion of the tools from posivity in algebraic geometry that we will use in this paper: Newton--Okounkov bodies and Zariski decomposition.

\subsection{Birational geometry} We start by recalling some relevant motivation and terminology from birational algebraic geometry. In this section, we let $\bK$ denote one of the rings $\Z,\Q$ or $\R$.

\vspace{3pt}

 A \emph{Weil} $\bK$\emph{-divisor} on a projective surface $Y$ is a formal $\bK$-linear combination of irreducible codimension one subvarieties of $Y$. We often omit \emph{Weil} from our language. We say that a Weil $\Z$-divisor $D$ on $Y$ is $\Q$\textit{-Cartier} if some integer multiple of $D$ is Cartier; that is, $D$ is the zero scheme of a section of a line bundle on $Y$. More generally, a Weil $\R$-divisor is $\Q$-Cartier if it can be expressed as an $\R$-linear combination of $\Q$-Cartier $\Z$-divisors. $Y$ is said to be $\Q$\textit{-factorial} if every Weil $\Z$-divisor on $Y$ is $\Q$-Cartier. 
 
 \begin{example} Every toric surface is $\Q$-factorial, as is every smooth projective surface. \end{example}
 
 \begin{notation} We fix the following notation for various groups and cones of divisors.
\begin{itemize}
\item $\on{Div}(Y)_\bK$ is the group of the Weil divisors on a $\Q$-factorial surface $Y$ with $\bK$-coefficients.
\vspace{3pt}
\item $\on{NS}(Y)_\bK:=\on{NS}(Y)\otimes_\Z\bK$ is the N\'eron--Severi group of $Y$, i.e. the group of integral Weil divisors up to algebraic equivalence tensored with $\bK$.
\vspace{3pt}
\item $\on{NE}(Y)_\bK$ is the cone of effective $\bK$-divisors inside the N\'eron--Severi group $\on{NS}(Y)_\bK$.
\vspace{3pt}
\item $\neb(Y)_\bK$ is the pseudo-effective cone of $Y$, i.e. the closure of the $\on{NE}(Y)_\bK$ in $\on{NS}(Y)_\bK$.
\vspace{3pt}
\item $\on{Nef}(Y)_\bK$ is the cone of divisor classes in $\on{NS}(Y)_\bK$ corresponding to nef divisors (namely, those divisors intersecting nonnegatively with every effective divisor)
\vspace{3pt}
\item $\on{Amp}(Y)_\bK$ is the cone of ample divisor classes in $\on{NS}(Y)_\bK$, i.e. those divisors intersecting positively with every effective divisor.
\vspace{3pt}
\item $\on{Big}(Y)$ is the group of big divisor classes, i.e. those in the interior of the effective cone.
\vspace{3pt}
\item $\R \cdot \on{Div}(Y)_\Z$ is the set of \emph{$\R \cdot \Z$-divisors} $D = c \cdot D'$ where $c \in \R$ and $D'$ is a $\Z$-divisor.
\end{itemize}
\end{notation}
\noindent There is a natural \emph{round-down} operation from $\R$-divisors to $\Z$-divisors that we denote by
\[\on{Div}(Y)_\R \to \on{Div}(Y)_\Z \qquad D = \sum_i a_i D_i \mapsto \lf D \rf := \sum_i \lf a_i \rf D_i\]
where $D_i$ are irreducible codimension one subvarieties of $Y$. 

\vspace{3pt}

\begin{notation} We let $\mO_Y(D)$ be the associated sheaf to a $\Z$-divisor (class) $D$ and $H^i(D) = H^i(\mO_Y(D))$ be the corresponding sheaf cohomology of dimension $h^i(D)$. 
\end{notation} 

\begin{definition}
Let $(Y,A)$ be a pair consisting of a $\Q$-factorial surface $Y$ and an $\R$-divisor $A$. We say that $(Y,A)$ is
\begin{itemize}
    \item a \emph{weakly polarized surface} if $A$ is big,
    \item a \emph{pseudo-polarized surface} if $A$ is big and nef,
    \item a \emph{polarized surface} if $A$ is ample.
\end{itemize}
Note that ample implies big and nef, so that polarized surfaces are pseudo-polarized and pseudo-polarized surfaces are weakly polarized.
\end{definition}

\subsection{Newton--Okounkov bodies} \label{sec:no} We next review the theory of Newton--Okounkov bodies \cite{oko_bru_96,lm_con_09,kk_new_12}, which are convex bodies that encode asymptotic information about sections of line bundles.

\vspace{3pt}

Fix a projective variety $Y$ and a valuation $\nu\colon\C(Y)\to\Z^n$. Note that we only require $\nu$ to be defined on $\C(Y)\setminus\{0\}$ to avoid compactifying $\Z^n$. The valuation restricts to a map
\[\nu\colon H^0(A) \subseteq \C(Y) \to \Z^n \qquad\text{for any divisor }A\]
Here we use the identification of global sections of $\mathcal{O}_Y(A)$ with rational functions with vanishing locus $A$. Let $\Delta_m(Y,A,\nu)$ denote the convex hull of the image of $H^0(mA)$ under $\nu$:
\begin{equation} \label{eqn:Delta_k}
\Delta_m(Y,A,\nu) := \on{conv}\big(\nu\left(H^0(mA)\right)\big)
\end{equation}

\begin{definition} \label{def:no_body}
The \emph{Newton--Okounkov body} $\Delta(Y,A,\nu)$ associated to $(Y,A,\nu)$ is defined by
\[\Delta(Y,D,\nu):=\ol{\bigcup_k \frac{1}{k} \cdot \Delta_k(Y,D,\nu)}\]
\end{definition}

\begin{remark} We note that, while explicit, Def.~\ref{def:no_body} does not assist much in the calculation of Newton--Okounkov bodies, which is typically very subtle. \end{remark}

\noindent Flags of subvarieties provide the main source of valuations. This will be the only type of valuation considered in this paper.

\begin{definition} A \emph{flag} $Y_\bullet$ in a projective variety $Y$ is a sequence of projective subvarieties
\[Y_0 \subseteq Y_1 \subseteq \dots \subseteq Y_n = Y \qquad \text{where}\qquad \on{dim}(Y_i) = i\]
A flag $Y_\bullet$ is \emph{locally smooth} if each $Y_i$ is smooth in a neighborhood of the point $Y_0$. We will assume this condition unless otherwise stated. \end{definition}

\begin{definition}
The valuation $\nu_{Y_\bullet}$ associated to a locally smooth flag $Y_\bullet$ is recursively given by the formula
$$\nu_{Y_\bullet}(s) := (\on{ord}_{Y_{n-1}}(s), \nu_{Z_\bullet}(s|_{Y_{n-1}}))$$
where $\nu_{Z_{\bullet}}(s|_{Y_{n-1}})$ is the valuation associated to the flag $Z_\bullet = Y_0 \subseteq \dots \subseteq Y_{n-1}$ on $Y_{n-1}$ evaluated on the rational function $s|_{Y_{n-1}}$.
\end{definition}

\begin{remark} When $Y$ is a surface every valuation arises in this way from a flag $\wt{Y}_\bullet$ on some birational model $\pi\colon\wt{Y}\to Y$ where $\pi$ is a series of blowups.
\end{remark}

\begin{notation} In the case of the valuation associated to a flag we will adopt the notation \[\Delta(Y,A,Y_\bullet) := \Delta(Y,A,\nu_{Y_\bullet}) \qquad \text{and}\qquad \Delta_m(Y,A,Y_\bullet) := \Delta_m(Y,A,\nu_{Y_\bullet})\]
\end{notation}

In the toric setting the Newton--Okounkov body generalizes the moment polytope by the following result.

\begin{prop}[\!{\cite[Prop.~6.1]{lm_con_09}}] Let $Y$ be a toric variety with torus-invariant ample divisor $A$ and moment polytope $\Omega$. Then for any torus-invariant flag $Y_\bullet$ we have
\[\Delta(Y,A,Y_\bullet) = \Omega\]
\end{prop}

\noindent One expects this since, essentially by definition, if $\Omega$ is the moment polytope for a polarized toric variety $(Y,A)$ then lattice points in $m\Omega$ correspond to torus-invariant sections in $H^0(\mO_Y(mA))$. Thus the rescaling by $\frac{1}{m}$ will just return $\Omega$ in the case that it is a lattice polytope. 

\subsection{Zariski decomposition} \label{sec:zariski} A related structure in algebraic positivity is Zariski decomposition, which splits a divisor on a surface into two pieces that each store positivity information.

\begin{definition}
Let $Y$ be a smooth surface and let $D$ be a pseudo-effective $\R$-divisor on $Y$. The \emph{Zariski decomposition} of $D$ is the unique expression
$$D=P+N$$
where $P$ is nef, $N$ is pseudo-effective, and $P\cdot N'=0$ for every irreducible component $N'$ of $N$. We call $P$ the \emph{positive part} and $N$ the \emph{negative part} of $D$.
\end{definition}

We state several properties of the Zariski decomposition into that we will use later.

\begin{prop} \label{prop:zariski_decomposition} Let $D$ be a pseudo-effective $\R$-divisor with Zariski decomposition
\[D = P + N\]
Then $P$ and $N$ have the following properties.
\begin{itemize}
    \item (Coefficients) If $D$ is a $\Z$ or $\Q$-divisor then $P$ and $N$ are $\Q$-divisors.
    \vspace{3pt}
    \item (Bigness) If $D$ is big then $P$ is also big.
    \vspace{3pt}
    \item (Maximality) $P$ is the largest nef $\R$-divisor such that $P \le D$. 
    \vspace{3pt}
    \item (Moduli)  The moduli of $D$ are captured by $P$ in the sense that
    \[h^0(mD)=h^0(mP) \qquad \text{for all }m \in \N\]
\end{itemize}
\end{prop}

\noindent The fourth property is a famous result of Zariski \cite{zar_the_62} and motivates using Zariski decomposition to study positivity questions.

\vspace{3pt}
The Zariski decomposition is also a powerful tool for studying volume. Recall that the \emph{volume} $\on{vol}(A)$ of a big $\R$-divisor $A$ on a variety $Y$ of dimension $n$ is given by
$$\on{vol}(A):=\lim_{m\to\infty}\frac{h^0(mA)}{m^n/n!}$$
This limit exists \cite[\S11.4]{laz_pos_17} and measures the leading asymptotics of section counts for multiples of $A$. The following lemma of Zariski is classical and is key to the interaction of Zariski decomposition and positivity.

\begin{lemma}[\!{\cite{zar_the_62}}] \label{lem:volume_zariski} Let $A$ be a big $\R$-divisor with Zariski decomposition $A=P+N$. Then $\on{vol}(A)=P^2$.
\end{lemma}

\subsection{Zariski chambers} \label{sec:chamber} The two positivity structures introduced above -- Newton--Okounkov bodies and Zariski decomposition -- interact through the notion of Zariski chambers. 

\vspace{3pt}

Introduced by Bauer--K\"uronya--Szemberg \cite{bks_zar_04}, Zariski chambers decompose the big cone of a smooth projective surface $Y$ in such a way that the Zariski decomposition is appropriately 'constant' within chambers. This also governs the points of nondifferentiability along the boundary of Newton--Okounkov bodies on $Y$; i.e.~their vertices. 

\vspace{3pt}

To define Zariski chambers, let $D = P + N$ be the Zariski decomposition of a pseudo-effective divisor so that $P$ is nef and $N$ is pseudo-effective. Denote the set of \emph{null curves} of $D$ by
\[\on{Null}(D):=\{C\text{ irreducible curve}:D\cdot C=0\}\]
Also denote the set of irreducible components of the negative part of $D$ by 
\[\on{Neg}(D):=\{\text{irreducible components of $N$}\}\]
Note that $\on{Neg}(D)\subseteq\on{Null}(P)$ by the construction of Zariski decomposition.

\begin{definition} The \emph{Zariski chamber} $\Sigma_P \subset \on{Big}(Y)$ of a big and nef divisor $P$ is defined as follows.
\[\Sigma_P :=\{D\in\on{Big}(Y):\on{Neg}(D)=\on{Null}(P)\}\]
The decomposition $\on{Big}(Y) = \bigsqcup_P \Sigma_P$ is called the \emph{Zariski decomposition} of the big cone of $Y$. \end{definition}

Intuituvely, a Zariski chamber consists of big divisors for which the support of their negative part is held constant. The chambers are convex cones that are typically neither open nor closed, and provide a locally finite cover of the big cone \cite[Lem.~1.4]{bks_zar_04}. 

\begin{example} \label{ex:-1_curves}
Suppose $Y$ is a rational surface whose only negative curves are $(-1)$-curves; i.e.~if $C\subseteq Y$ has $C^2<0$ then $C^2=-1$. This class includes all del Pezzo surfaces. The Zariski chambers for $Y$ are described by the restriction of the hyperplane arrangement $\{C^\perp\}_{C^2<0}$ consisting of the orthogonal complements $C^\perp:=\{D\in\on{Pic}(Y)_\R:D\cdot C=0\}$ for each irreducible negative curve on $Y$ to the big cone \cite[Prop.~3.4]{bks_zar_04}.
\end{example}

The Zariski chamber decomposition controls the boundary of Newton--Okounkov bodies of surfaces in very precise terms due to the foundational work of Lazarsfeld--Musta\c{t}\u{a} \cite{lm_con_09}. Precisely, fix a projective surface $Y$, a flag $Y_\bullet = \{y \in Y_1 \subset Y\}$ and a big divisor $A$. Consider the divisors
\[A_t := A - tY_1 \qquad \text{with Zariski decomposition} \qquad A_t = N_t + P_t \]
Let $\mu$ be the supremum over $t$ such that $A_t$ is in the big cone. 

\begin{thm}[\!{\cite[Thm.~6.4]{lm_con_09}}] The Newton-Okounkov body $\Delta(Y_\bullet,A)$ is given by
\[
\Delta(Y,A,Y_\bullet) = \big\{(t,h)\in\R^2 \;:\;0\leq t\leq\mu\quad\text{and}\quad\alpha_{Y_\bullet}(A,t)\leq h\leq\beta_{Y_\bullet}(A,t)\big\}
\]
where $\alpha_{Y_\bullet}(A,\cdot)$ and $\beta_{Y_\bullet}(A,\cdot)$ are continuous and piecewise linear in $t$ on any interval $[0,\mu']$ with $\mu' < \mu$. 
\end{thm}

The corners (i.e. points of nondifferentiability) of $\alpha_{Y_\bullet}$ and $\beta_{Y_\bullet}$ are governed by the Zariski chamber decomposition. Morally, corners occur precisely at values of $t$ where $A_{t+\eps}$ lies in a different Zariski chamber to $A_t$ for all small $\eps>0$, which the following result makes precise.

\begin{lemma}[\!{\cite[Thm.~B]{klm_con_12}}] With the setup above we have:
\begin{itemize}
    \item the map $\alpha_{Y_\bullet}(A,\cdot)$ is not differentiable at $t$ if and only if, for every $\epsilon > 0$, there is a curve $C$ with
    \[C \in \on{Neg}(A_{t+\epsilon}) \setminus \on{Neg}(A_t) \qquad\text{and}\qquad y \in C\]
    \item the map $\beta_{Y_\bullet}(A,\cdot)$ is not differentiable at $t$ if and only if, for every $\epsilon > 0$, there is a divisor $C$ with
    \[C \in \on{Neg}(A_{t+\epsilon}) \setminus \on{Neg}(A_t) \qquad\text{and}\qquad y \not\in C\]
\end{itemize}
\end{lemma}

\begin{remark} The Zariski decomposition is locally finite and bears many similarities to the natural chamber decomposition arising from algebraic capacities \cite[\S3]{wor_alg_22}; e.g.~in \cite[Prop.~1.14]{bks_zar_04} the local finiteness of the Zariski chamber decomposition is used to prove continuity for Zariski decomposition while in \cite[Cor.~3.2]{wor_alg_22} the local finiteness of the chamber decomposition from algebraic capacities is used to show their continuity. Understanding any relationship between the two -- especially in light of the results of \S\ref{sec:zariski_methods} -- is an interesting problem. \end{remark}

\section{Symplectic embeddings } \label{sec:symp_emb} In this section, we develop a variation of a construction of certain symplectic embeddings that uses deformations (i.e. families of projective varieties over the complex line) and Newton--Okounkov bodies, due to Kaveh \cite{kav_tor_19}. We will refer to these as \emph{Kaveh embeddings}. 

\subsection{Basic notions} Recall that a \emph{symplectic manifold} $X = (X,\omega)$ is a smooth manifold $X$ equipped with a closed non-degenerate $2$-form $\omega$, called a \emph{symplectic form}.

\begin{definition} A \emph{symplectic embedding} $\iota\colon X \to X'$ between symplectic manifolds $(X,\omega)$ and $(X',\omega')$ of the same dimension is a smooth embedding that satisfies $\iota^*\omega' = \omega$. We write 
\[X \se X'\]
when a symplectic embedding $X \to X'$ exists, without specifying the map. \end{definition}

The main tools for obstructing the existence of symplectic embeddings are symplectic capacities. A \emph{symplectic capacity} $c$ is a numerical invariant of some class of symplectic manifolds such that
\[c(X,\omega) \le c(X',\omega') \quad\text{if}\quad X \se X' \qquad\text{and}\qquad c(X,a\omega) = a \cdot c(X,\omega)\]
Typically an additional normalization constraint is included to rule out trivial capacities.

\vspace{3pt}

We will primarily focus on the \emph{ECH capacities} of a symplectic 4-manifold $(X,\omega)$ originally introduced by Hutchings using embedded contact homology \cite{hut_qua_11,hut_lec_14}. We denote these capacities by
\[\cech_k(X,\omega) \in (0,\infty] \qquad \text{for}\qquad k \in \N\]
These capacities have many useful properties beyond the usual axioms of capacities. Most pertinent to this paper is the following \emph{volume property}.

\begin{thm}[\!{\cite[Thm.~1.1]{chr_asy_15}}] \label{thm:volume_property} Let $(X,\omega)$ be a Liouville domain. Then
\[\lim_{k\to\infty}\frac{\cech_k(X,\omega)^2}{k}=4\on{vol}(X,\omega)\]
\end{thm}

\noindent Note that Thm.~\ref{thm:volume_property} extends to closed manifolds $X$ that admit symplectic embeddings from Liouville domains $W$ with volume arbitrarily close that of $X$. We can capture the sub-leading asymptotic behavior of the ECH capacities by defining
$$e_k(X,\omega):=\cech_k(X,\omega)-\sqrt{4\on{vol}(X,\omega)k}$$
There has been much study of the asymptotics of these error terms via symplectic methods \cite{cs_sub_20,sun_est_19,hut_ech_22} and algebraic methods \cite{wor_alg_22,wor_tow_22}.

\subsection{Toric domains} There is a rich class of symplectic manifolds that will be particularly important in this paper. Consider the standard moment map
\[\mu\colon\C^n \to \R_{\geq0}^n \qquad \mu(z_1,\dots,z_n) = \pi \cdot (|z_1|^2,\dots,|z_n|^2)\]

\begin{definition} \label{def:moment_domain} A \emph{moment domain} $\Omega \subseteq\R_{\geq0}^n$ is any domain containing a neighborhood of the origin. The associated \emph{toric domain} $X_\Omega$ and \emph{open free domain} $F_\Omega$ are given by
\[X_\Omega := \mu^{-1}(\Omega) \quad\text{and}\quad F_\Omega := \mu^{-1}(\intr{\Omega})\]
Here $\intr{\Omega}$ denotes the interior of $\Omega$ as a subset of $\R^n$. We say that $\Omega$ is
\begin{itemize}
    \item \emph{(weakly) convex} if $\Omega$ is convex as a subset of $\R_{\geq0}^n$.
    \vspace{2pt}
    \item \emph{concave} if the complement of $\Omega$ is convex as a subset of $\R_{\geq0}^n$.
    \vspace{2pt}
    \item \emph{integral} if it is the convex hull of a finite set points in $\Z^n$.
    \vspace{2pt}
    \item \emph{rational} if $m \cdot \Omega$ is integral for some $m \in \N$.
    \item \emph{rational-sloped} if it is a polytope such that there is an integral vector normal to each facet.
\end{itemize}\end{definition}

\begin{example}
If $\Omega=\Delta(a)$, the triangle with vertices $(0,0),(a,0),(0,a)$, then $X_\Omega=B^4(a)$. More generally, if $\Omega$ is the triangle with vertices $(0,0),(a,0),(0,b)$ then $X_\Omega$ is the ellipsoid with symplectic radii $a$ and $b$ given by
$$X_\Omega=E(a,b):=\left\{(z_1,z_2)\in\C^2:\frac{\pi|z_1|^2}{a}+\frac{\pi|z_2|}{b}\leq 1\right\}$$
\end{example}

\begin{figure}[h]
\caption{Concave and convex moment domains}
\label{fig:domains}

\begin{center}
\begin{tikzpicture}[scale=0.6]

\draw[fill=gray!30] (0,0) to (0,5) to (1,2) to (2,1) to (4,0) to (0,0);

\draw[fill=gray!30] (8,0) to (8,4) to (10,3) to (12,1) to (12,0) to (8,0);
\draw (8,4) to (8,5);
\draw (12,0) to (13,0);

\foreach \i in {0,...,4}
{
\foreach \j in {0,...,5}
{\node (\i\j) at (\i,\j){\tiny $\bullet$};
\node (\i\j) at (8+\i,\j){\tiny $\bullet$};
}
\node (\i) at (13,\i){\tiny $\bullet$};
}

\node (5) at (13,5){\tiny $\bullet$};

\node (la) at (2,-0.8){(a)};
\node (lb) at (10.5,-0.8){(b)};
\end{tikzpicture}
\end{center}
\end{figure}
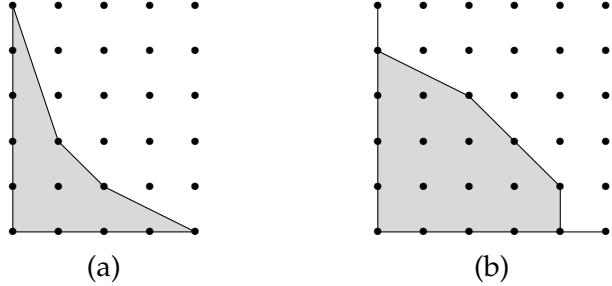

Toric domains are related to toric varieties from algebraic geometry as follows. Let $\Omega \subseteq \R^n$ be an integral polytope; i.e. a polytope whose vertices lie on $\Z^n$. Let
\[
S := \Omega \cap \Z^n \subset \Z^n
\]
Choose an ordering $s_1,\dots,s_r$ of the elements of $S$ and consider the map
\[\iota_S:\C^n \to \P^{r-1} \qquad \text{given by}\qquad \iota_S(z_1,\dots,z_n) = [z^{s_1},z^{s_2},\dots,z^{s_r}] \in \P^{r-1}\]
Here $z^s$ for $s \in \Z^n$ denotes the product
\[z^s = z_1^{s_1}\dots z_n^{s_n}\]
The closure of $\iota_S(\C^n)$ is precisely the toric variety $Y_\Omega$ associated to the polytope $\Omega$ and the ample divisor $A_\Omega$ is the intersection of this closure with the hyperplane at infinity $\P^{r-2} \subseteq \P^{r-1}$ \cite[Ch.~2]{cls_new_15}. The link between the toric domain and toric variety is given by the following.

\begin{lemma}[\!{\cite[\S1.2]{cw_ech_20}}] \label{lem:toric_spaces_equal} Let $\Omega$ be a moment domain that is a convex lattice polytope and let $\iota_S\colon\C^n \to \P^{r-1}$ be the corresponding map to projective space. Then
\[
\on{int}(X_\Omega) \simeq \iota_S(\C^n) = Y_\Omega \setminus A_\Omega
\]
\end{lemma}

The ECH capacities of toric domains are extremely well understood and are governed by the so-called ``weight sequence'' of the moment domain $\Omega$. We start with the concave case from \cite{ccfhr_sym_14}.

\begin{definition} \label{def:ccv_weight_seq} The \emph{weight sequence} $\on{wt}(\Omega)$ of a concave moment domain $\Omega\subseteq\R_{\geq0}^2$ is the (potentially infinite) unordered list (with repetitions) defined inductively as follows. 

\vspace{3pt}

Let $\Omega_1 = \Delta(a)$ be the largest equilateral right triangle contained in $\Omega$. The complement $\Omega \setminus \Omega_1$ consists of two (possibly empty) components: a component $\Omega'_2$ touching the $x$-axis and a component $\Omega'_3$ touching the $y$-axis. It is simple to check that $\Omega_2 = T_2\Omega_2' + v_2$ and $\Omega_3 = T_3\Omega_3' + v_3$ are concave moment domains, where
\[
T_2 = \left[\begin{array}{cc}
1 & 0\\
1 & 1\\
\end{array}\right] \qquad 
T_3 = \left[\begin{array}{cc}
1 & 1\\
0 & 1\\
\end{array}\right] \qquad 
v_2 = -\left[\begin{array}{cc}
0\\
a\\
\end{array}\right] \qquad
v_3 = -\left[\begin{array}{cc}
a\\
0\\
\end{array}\right]
\]
The weight expansion $\on{wt}(\Omega)$ is now defined by the following recursive formula.
\[
\on{wt}(\Omega) = \{a\} \cup \on{wt}(\Omega_1) \cup \on{wt}(\Omega_2)
\]
The first stage of the process for the concave domain from Fig.~\ref{fig:domains}(a) is shown in Fig.~\ref{fig:wt_seq}(a).\end{definition}

\begin{definition} \label{def:cvx_weight_seq} The \emph{weight sequence} $\on{wt}(\Omega)$ of a convex moment domain $\Omega\subseteq\R_{\geq0}^2$ is the sequence
\[
\on{wt}(\Omega) = (a;\on{wt}(\Omega_1),\on{wt}(\Omega_2))
\]
Here $\Delta(a)$ is the smallest right triangle containing $\Omega$, and $\Omega_1$ and $\Omega_2$ are the (possibly empty) concave domains constructed from the components of $\Delta(a) \setminus \Omega$ as in the concave case. We show the first step of the process for the convex domain from Fig.~\ref{fig:domains}(b) in Fig.~\ref{fig:wt_seq}(b).
\end{definition}

\begin{remark} As discussed in \cite[\S2.4]{wor_tow_22} it is most natural to view this sequence as indexed by a rooted binary tree. \end{remark}

\begin{remark}
A concave or convex moment domain $\Omega$ is rational-sloped (for instance, rational) if and only if its weight sequence is finite.
\end{remark}

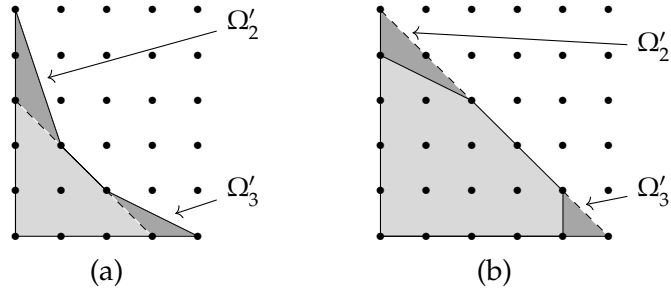
\begin{figure}[h]
\caption{Weight decompositions}
\label{fig:wt_seq}
\begin{center}
\begin{tikzpicture}[scale=0.6]
\draw[fill=gray!30] (0,0) to (0,5) to (1,2) to (2,1) to (4,0);
\draw[white,fill=gray!70] (0,3) to (1,2) to (0,5) to (0,3);
\draw[white,fill=gray!70] (3,0) to (2,1) to (4,0) to (3,0);
\draw (0,0) to (0,5) to (1,2) to (2,1) to (4,0);

\draw[white,fill=gray!30] (8,0) to (8,4) to (10,3) to (12,1) to (12,0);
\draw[white,fill=gray!70] (8,4) to (8,5) to (10,3);
\draw[white,fill=gray!70] (12,0) to (13,0) to (12,1);
\draw (8,0) to (8,5);
\draw (8,0) to (13,0);
\draw (8,4) to (10,3) to (12,1) to (12,0);

\foreach \i in {0,...,4}
{
\foreach \j in {0,...,5}
{\node (\i\j) at (\i,\j){\tiny $\bullet$};
\node (\i\j) at (8+\i,\j){\tiny $\bullet$};
}
\node (\i) at (13,\i){\tiny $\bullet$};
}

\node (5) at (13,5){\tiny $\bullet$};

\draw (0,0) to (4,0);
\draw[dashed] (0,3) to (3,0);

\node (l2) at (5,1){\small $\Omega_3'$};
\node (l1) at (5,4.7){\small $\Omega_2'$};

\draw[->] (l2) to (3.5,0.5);
\draw[->] (l1) to (0.7,3.5);

\draw (8,0) to (12,0);
\draw[dashed] (8,5) to (10,3);
\draw[dashed] (12,1) to (13,0);

\node (l2) at (14,1){\small $\Omega_3'$};
\node (l1) at (14,4.2){\small $\Omega_2'$};

\draw[->] (l2) to (12.5,0.7);
\draw[->] (l1) to (8.7,4.5);

\node (la) at (2,-0.8){(a)};
\node (lb) at (10.5,-0.8){(b)};
\end{tikzpicture}
\end{center}
\end{figure}

ECH capacities can be used to completely characterize the existence of symplectic embeddings of a concave toric domain to a convex toric domain \cite{cri_sym_19} due to a result of Cristofaro-Gardiner.

\begin{prop}[\!{\cite[Thm. 1.2]{cri_sym_19}}] \label{prop:cg_embedding} Let $\Omega$ and $\Delta$ be concave and convex moment domains respectively in $\R_{\geq0}^2$. The following are equivalent:
\begin{enumerate}
    \item The interior of $X_\Omega$ symplectically embeds into the interior of $X_\Delta$.
    \vspace{3pt}
    \item For each $c$ with $0 < c < 1$, $X_{c\Omega}$ symplectically embeds into the interior of $X_\Delta$. 
    \vspace{5pt}
    \item $\cech_k(X_\Omega) \le \cech_k(X_\Delta)$ for each $k \ge 1$.
\end{enumerate}
\end{prop}

\subsection{Kaveh deformations} \label{subsec:Kahler_deformations} Fix a smooth projective variety $Y$ of dimension $n$ equipped with a flag $Y_\bullet$. In \cite{kav_tor_19} Kaveh constructed a K\"ahler manifold equipped with a submersion to $\C$ with generic fiber $Y$ and central fiber symplectomorphic to an open free domain $F_\Omega$.

\vspace{3pt}

The goal of this subsection is to describe a version of a Kaveh's construction where the central fiber is a toric domain $X_\Omega$. The construction is essentially the same as in \cite{kav_tor_19} but we carefully check that certain details (e.g.~the extension of holomorphic functions) carry over to our setting.

\subsubsection{Taylor series valuation} We start by describing a local version of the valuation of a flag and describing some of its properties. Let $f$ be a formal Taylor series in $n$ variables, written as
\begin{equation} \label{eqn:f_Taylor_series}  f(u_1,\dots,u_n) = \sum_\alpha c_\alpha u^\alpha \qquad \text{where}\qquad \alpha = (\alpha_1,\dots,\alpha_n) \in \Z_{\ge 0}^n\end{equation}
Recall that the \emph{support} $\on{supp}(f)$ is the set of integer vectors such that the coefficient $c_\alpha$ is non-zero.
\[\on{supp}(f) := \{\alpha \; : \; c_\alpha \neq 0\} \subseteq \Z^n_{\ge 0}\]
There is a natural \emph{order valuation} $\nu_{\on{TS}}$ on the fraction field of $\C[[u_1,\dots,u_n]]$ determined by the standard lexicographic ordering $\prec$ on $\Z^n$. On a formal Taylor series, it is given by
\[
\nu_{\on{TS}}(f) = \on{min}\big(\alpha \; :\; \alpha \in \on{supp}(f)\big)
\]
The valuation $\nu_{Y_\bullet}$ of a flag $Y_\bullet$ is the pullback of $\nu_{\on{TS}}$ through any coordinate chart $\varphi\colon V \simeq U \subseteq \C^n$ on a neighborhood $V \subseteq Y$ that satisfies
\begin{equation} \label{eqn:flag_chart_standard}
\varphi(V \cap Y_k) = U \cap (0 \times \C^k) \subseteq \C^n
\end{equation}
More precisely, the following diagram commutes, where the horizontal map sends $f$ to the Taylor series of $f \circ \varphi^{-1}\colon U \to \C$. 
\[
\begin{tikzcd}
\big\{f:X \to \C \text{ holomorphic near }Y_0\big\} \arrow[r] \arrow[d,"\nu_{Y_\bullet}"] & \C[[u_1,\dots,u_n]] \arrow[d,"\nu_{\on{TS}}"]\\
\Z^n \arrow[r,"="] & \Z^n
\end{tikzcd}
\]

We will require the following property of the order valuation, which roughly states that the valuation of $f$ is a boundary point of the convex hull of the support.

\begin{lemma} \label{lem:valuation_w} There is a $C > 0$ such if $w = (w_1,\dots,w_n)$ satisfies $w_i \le 0$ and $w_{i+1} \le Cw_i$ for all $i$, then $w \cdot \nu_{\on{TS}}(f) \le w \cdot \alpha$ for all $\alpha \in \on{supp}(f)$. \end{lemma}

\begin{proof} Let $v = \nu_{\on{TS}}(f)$ and $r_i = v_i + v_{i+1} + \dots + v_n$. Choose $w = (w_1,\dots,w_n)$ to satisfy
\begin{equation} \label{eqn:w_ratio}
w_i < 0 \quad\text{and}\quad w_i \ge C \cdot w_{i+1} \qquad\text{where}\qquad C = \max{1,r_1,r_2,\dots,r_n}
\end{equation}
Now let $\alpha \in \on{supp}(f)$. Since $v = \nu_{\on{TS}}(f)$, there is a $j$ such that $\alpha_j = v_j$ for $j < i$ and $\alpha_i \ge v_i + 1$. Thus we have
\[
w \cdot (\alpha - v) = \sum_{j < i} w_j(\alpha_j - v_j) + w_i(\alpha_i - v_i) + \sum_{k > i} w_k(\alpha_k - v_k) \le w_i - \sum_{k > i} w_kv_k
\]
Due to (\ref{eqn:w_ratio}), we have $w_{i+1} \le w_{j}$ for all $j \ge i+1$ and thus
\[w_i - \sum_{k > i} w_kv_k \le w_i - w_{i+1}r_{i+1} \le 0 \qedhere\] \end{proof}

\noindent As an immediate consequence of Lem.~\ref{lem:valuation_w} we have the following corollary.

\begin{cor} \label{cor:valuation_w} If $f_1,\dots,f_r$ are a set of formal Taylor series, then there exists a $w \in \N^n$ such that
\[w \cdot \nu_{\on{TS}}(f_i) \le w \cdot \alpha \qquad\text{for all}\qquad \alpha \in \on{supp}(f_i)\]
\end{cor}

\noindent Analogous results to Lem.~\ref{lem:valuation_w} and Cor.~\ref{cor:valuation_w} also hold for the valuation $\nu_{Y_\bullet}$. 

\subsubsection{Weighted deformations} We next describe a construction of a certain projective variety
\[\mathcal{X}(Y_\bullet,m)\]
determined by $Y_\bullet$ and a choice of integer vector $m \in \N^n$. Fix a chart $\varphi\colon V \simeq U \subseteq \C^n$ on $Y$ that satisfies the compatibility condition (\ref{eqn:flag_chart_standard}) with the flag. To define this space, consider the map $\Phi_m\colon\C^\times \times \C^n \to \C \times \C^n$ using the coordinates $u_1,\dots,u_n$ defined as follows.
\[
\Phi_m(t;u_1,\dots,u_n) = (t;t^{-m_1}u_1,\dots,t^{-m_n}u_n)
\]
We also define $W = \Phi_m(\C^\times \times U) \cup (0 \times \C^n) \subseteq \C \times \C^n$ and let $\Psi$ denote the composition map
\[\C^\times \times V \xrightarrow{\on{Id} \times \varphi} \C^\times \times U \xrightarrow{\Phi_m} W\]

\begin{lemma} The set $W \subseteq \C \times \C^n$ is open and $\Psi$ is a biholomorphism of $\C^\times \times V$ with an open subset of $W$ commuting with projection to $\C$.
\end{lemma} 

\begin{proof} The only non-trivial part of the claim is that $W$ is open. Since $\Psi$ is a local biholomorphism, it suffices to check that every $(0;z) = (0;z_1,\dots,z_n) \in 0 \times \C^n$ is an interior point of $W$. Choose any ball $B(z) \subseteq \C^n$ around $z$ and note that $w = (t^{m_1}w_1,\dots,t^{m_n}w_n)$ is in $U$ for $w \in B(z)$ and $|t|$ sufficiently small, since $U$ is a neighborhood of $0$. It follows that for sufficiently small $\delta$
\[
[-\delta,\delta]^2 \times B(z) \subseteq W \qedhere
\]
\end{proof}

\begin{definition} The \emph{weighted deformation to the normal cone} $\mathcal{X} = \mathcal{X}(Y_\bullet,m)$ associated to $(Y_\bullet,m)$ is given by the gluing $(\C^\times \times Y) \cup_\Psi W$, equipped with the holomorphic submersion
\[
\begin{tikzcd}
\C^\times \times Y \arrow[r] \arrow[d] & \mathcal{X}(Y_\bullet,m) \arrow[d,"\pi"]\\
\C^\times \arrow[r] & \C
\end{tikzcd} \qquad\text{with}\qquad \pi^{-1}(0) = T_{Y_0}Y \simeq \C^n
\]\end{definition}

\subsubsection{Holomorphic functions on deformations}  Let $f\colon Y \to \C$ be a meromorphic function that is holomorphic near $Y_0$ and let $\on{supp}(f \circ \varphi^{-1})$ be the support of the Taylor series of $f \circ \varphi^{-1}\colon U \to \C$.

\begin{definition} \label{def:compatible_pair} A pair $(m,v)$ of an integer vector $m \in \N^n$ and a vector $v \in \on{supp}(f \circ \varphi^{-1})$ is \emph{compatible} if $v$ is the unique minimum of the function
\[
\on{supp}(f \circ \varphi^{-1}) \to \Z \qquad \alpha \mapsto \alpha \cdot m
\]
\end{definition}

\begin{lemma} \label{lem:Kaveh_function_extension} If $(m,v)$ is a compatible pair, then the function $t^{-m \cdot v} f$ on $\C^\times \times Y \subseteq \mathcal{X}(Y_\bullet,m)$ extends holomorphically over the complement $0 \times \C^n$ of $\C^\times \times Y$ to a function $F$ satisfying
\begin{equation} \label{eq:Kaveh_function_extension} F(0,w) = c_v w^v \qquad \text{for}\qquad (u,w) \in 0 \times \C^n\end{equation}
\end{lemma}

\begin{proof} Consider the coordinates $t$ and $w_i = t^{-m_i} u_i$ on the neighborhood $W \subset \mathcal{X}(Y_\bullet,m)$ of $0 \times \C^n$. In terms of the notation in (\ref{eqn:f_Taylor_series}), the Taylor series for $f$ in the coordinates $(t,w)$ on $W$ is given by
\[
f(t;w) = \sum_\alpha c_\alpha t^{m \cdot \alpha}w^\alpha = \sum_\alpha c_\alpha \Big(\prod_i (t^{m_i}w_i)^{\alpha_i}\Big)
\]
Since the Taylor series (\ref{eqn:f_Taylor_series}) is absolutely convergent for small $|u|$, the Taylor series of $t$ and $w$ are absolutely convergent for any $w \in \C^n$, as long as $|t|$ is sufficiently small. On the other hand, the Taylor series is divisible by $t^{m \cdot v}$ since $m \cdot v \le m \cdot \alpha$ whenever $c_\alpha \neq 0$. Therefore
\begin{equation} \label{eqn:taylor_expansion_tf}
t^{-m \cdot v} f(t;w) = \sum_\alpha c_\alpha t^{m \cdot (\alpha - v)}w^\alpha\end{equation}
is absolutely convergent in a neighborhood of $(0,w) \in 0 \times \C^n$, for any $w$. 

\vspace{3pt}

Thus $t^{-m \cdot v} f$ extends holomorphically over $0 \times \C^n$ and is given by the Taylor series (\ref{eqn:taylor_expansion_tf}) near $0 \times \C^n$. In particular, when $t = 0$, the expansion only contains terms $c_\alpha w^\alpha$ with $m \cdot \alpha = m \cdot v$. Since $v$ is the unique minimizer of the map $\alpha \mapsto m \cdot \alpha$, we acquire (\ref{eq:Kaveh_function_extension}). 
\end{proof}

\begin{definition} \label{def:meromorphic_ext_to_family} The meromorphic map $F(m,v)\colon\mathcal{X}(Y_\bullet,m) \to \C$ associated to a meromorphic map $f\colon Y \to \C$ and a compatible pair $(m,v)$ is defined uniquely by
\[
F(m,v)(t;z) = t^{-m \cdot v} \cdot f(z) \qquad\text{on}\qquad \C^\times \times Y \subseteq \mathcal{X}(Y_\bullet,m)
\]
\end{definition}

\subsubsection{Projective embeddings of deformations} Let $A$ be a very ample divisor on $Y$ and consider the corresponding Kodaira embedding
\[
\phi\colon Y \to \P E \qquad \text{where} \qquad E = (H^0(A))^\vee
\]
Choose a basis of $E$ and consider the dual basis of sections of the associated line bundle $L$.
\[\sigma_1,\dots,\sigma_r \in H^0(A) \simeq \Gamma(L)\]
We may assume that $\tau = \sigma_r$ is non-vanishing at the basepoint $Y_0$ of the flag $Y_\bullet$. For each section $\sigma_i$, we have a meromorphic function given by
\[
\sigma_i = f_i \cdot \tau
\]
After rescaling $\sigma_i$, we may assume that the Taylor coefficient of $f_i$ corresponding to index $\nu_{Y_\bullet}(f_i) = \nu_{Y_\bullet}(\sigma_i)$ is $1$. We let $\nu(\sigma)$ denote the set
\[\nu(\sigma) := \big\{\nu_{Y_\bullet}(\sigma_i) \; : \; i = 1,\dots,r\big\} \qquad\text{and}\qquad \Delta(\sigma) := \on{conv}\big(\nu(\sigma)\big) \subseteq \R_{\geq0}^n\]
Finally, use Cor.~\ref{cor:valuation_w} to choose $m \in \N^n$ such that $(m,\nu_{Y_\bullet}(\sigma_i))$ is a compatible pair for each $i$.

\begin{definition} The extension $\Phi(\sigma)\colon\mathcal{X}(Y_\bullet,m) \to \C \times \P^{r-1}$ of the Kodaira map $\phi$ to $\mathcal{X}(Y_\bullet,m)$ is defined as the map
\[
(t;z) \mapsto (t;[F_1(t;z),\dots,F_r(t;z)]
\]
where $F_i = F(m,\nu_{Y_\bullet}(\sigma_i))$ is the meromorphic function determined by $f_i$ in Definition \ref{def:meromorphic_ext_to_family}. Note that
\begin{equation} \label{eqn:toric_map_Phi}
\Phi(\sigma)(0;w_1,\dots,w_n) = [w^{v_1},w^{v_2},\dots,w^{v_{r-1}},1] \qquad\text{where}\qquad v_i = \nu_{Y_\bullet}(\sigma_i)
\end{equation}
\end{definition}

\begin{lemma}\label{lem:domain_implies_symplectic} If $\Delta(\sigma) \subseteq \R_{\geq0}^n$ is a moment domain (see Def.~\ref{def:moment_domain}) then $\Phi(\sigma)$ is an embedding. 
\end{lemma}

\begin{proof} It suffices to check that the restriction of $\Phi(\sigma)$ to $0 \times \C^n$ is an an embedding and that the differential is full rank there. If $\Delta(\sigma)$ is a moment domain then some multiples of the unit basis vectors $e_1,\dots,e_n$ are in $\nu(\sigma)$. From \cite[Cor.B(ii)]{kl_loc_18} and the very ampleness of $A$ we see that in fact $e_1,\dots,e_n$ must be in $\nu(\sigma)$. After reordering, we may assume that $v_i = e_i$ for $i = 1,\dots,n$. Under this assumption, we can write $\Phi(\sigma)|_{0 \times \C^n}$ as
\[\C^n \xrightarrow{\on{Id} \oplus F} \C^n \oplus \C^{r-n-1} \quad\text{in the chart}\quad (z_1,\dots,z_{r-1}) \mapsto [z_1,\dots,z_{r-1},1]\]
Here $F\colon\C^n \to \C^{r-n-1}$ is holomorphic. Thus $\Phi(\sigma)$ is a injective and has full rank differential.
\end{proof}

\subsubsection{Kahler forms on deformations} Let $\Omega(\sigma)$ be the pullback of the Fubini--Study form on $\P^{r-1}$. 
\[\Omega(\sigma) := \Phi(\sigma)^*\Omega_{\on{FS}}\]
Moreover, let $(Y_s,\omega_s)$ be the fiber $\pi^{-1}(s)$ equipped with the $2$-form $\Omega(\sigma)|_{Y_s}$.

\begin{lemma} If $\Delta(\sigma)$ is a moment domain, then $\Omega(\sigma)$ is a symplectic form on the fibers $Y_s$. Furthermore
\begin{itemize}
    \item[(a)] $(Y_0,\omega_0)$ is symplectomorphic to the interior of the toric domain $X_{\Delta(\sigma)}$
    \item[(b)] $(Y_s,\omega_s)$ is symplectomorphic to $(Y,\omega_A)$ for $s \neq 0$.
\end{itemize}\end{lemma}
 
\begin{proof} The first claim is immediate from Lem.~\ref{lem:domain_implies_symplectic}. Moreover, (\ref{eqn:toric_map_Phi}) and Lem.~\ref{lem:toric_spaces_equal} immediately imply (a). Finally, (b) follows from the fact that
\[
\omega_s = \psi^*\Omega_{\on{FS}}
\]
where $\psi\colon Y \to \P^{r-1}$ is the Kodaira embedding determined by the sections $\tau_i = s^{-v_i \cdot m} \sigma_i$ where $v_i = \nu_{Y_\bullet}(\sigma_i)$. On the otherhand, the symplectic form $\omega_A$ is defined (up to symplectomorphism) as the pullback of $\Omega_{\on{FS}}$ by the Kodaira embedding determined by \emph{any} basis. 

\vspace{3pt}

Note that $\omega_A$ is well-defined (up to symplectomorphism) because the space of bases of sections of $A$ is path connected, and so any two such pullbacks are connected by a cohomologous path of symplectic forms. In particularly, they are symplectomorphic by the Moser trick.   \end{proof}

\begin{lemma} \label{lem:embedding_from_family} If $\Delta(\sigma)$ is a moment domain, then for any compact subset $X \subseteq Y_0 \simeq \on{int}(X_{\Delta(\sigma)})$, there is a symplectic embedding $X \to (Y,\omega_A)$.
\end{lemma}

\begin{proof} We use the gradient Hamiltonian flow as in \cite[\S 7]{kav_tor_19} on $\mathcal{X} = \mathcal{X}(Y_\bullet,m)$. Consider the real part $h\colon\mathcal{X} \to \R$ of the projection $\pi:\mathcal{X} \to \C$ and consider the following vector-field $V$ on $\mathcal{X}$. 
\[V := \frac{\nabla h}{\langle \nabla h,\nabla h\rangle}\]
Here $\langle-,-\rangle$ is the associated Riemannian metric to $\Omega(\sigma)$, and $\nabla$ is the gradient vector-field with respect to this metric. For sufficiently small $t$, we have a flow
\[\Psi\colon[0,t] \times X \to \mathcal{X} \qquad\text{with}\qquad \Psi_0|_X = \on{Id}\]
By \cite[Prop 7.1]{kav_tor_19} $\Psi_s$ is a symplectic embedding to $Y_s$ which is symplectomorphic to $(Y,\omega_A)$.   
\end{proof}

\subsection{Construction of Kaveh embeddings} We conclude this section by applying \S\ref{subsec:Kahler_deformations} to the construction of symplectic embeddings from toric domains to projective varieties.

\vspace{3pt}

Let $(Y,A)$ be a smooth polarized variety equipped with a flag $Y_\bullet$. Recall that the Newton--Okounkov body of $(Y,A)$ with respect to $Y_\bullet$ is defined by
\[
\Delta(Y,A,Y_\bullet) = \lim_{m \to \infty} \frac{1}{m} \cdot\Delta_m(Y,A,Y_\bullet) \qquad\text{where}\qquad \Delta_m(Y,A,Y_\bullet) = \text{conv}\big(\nu_{Y_\bullet}(H^0(mA))\big)
\]
The limit here is a Hausdorff limit of sets.

\begin{prop}[Kaveh Embeddings] \label{thm:Kaveh_embeddings}  Let $(Y,A)$ be a smooth, polarized variety with a flag $Y_\bullet$ and let $X \subset \intr{X_\Omega}$ be a subset of the interior of the toric domain with moment polytope
\[\Omega = \frac{1}{m}\Delta_m(Y,A,Y_\bullet) \qquad\text{for some }m\]Then there is a symplectic embedding $X \to (Y,\omega_A)$. 
\end{prop}

\begin{proof} After passing to a multiple of $m$ we may assume that $mA$ is very ample and that $m \Omega \subseteq \on{int}(\Delta_m(Y,A,Y_\bullet))$. Pick a basis $\sigma$ of sections of $H^0(mA)$, so that $\Delta_m(Y,A,Y_\bullet) = \Delta(\sigma)$. By our hypothesis, $\Delta(\sigma)$ is a moment domain and $\sqrt{m} \cdot X \subseteq \on{int}(X_{\Delta(\sigma)})$. Here $\sqrt{m} \cdot X$ denotes the scaling of $X$ as a subset of $\C^n$. By Lem.~\ref{lem:embedding_from_family} we have a symplectic embedding
\[
\sqrt{m} \cdot X \to (Y,\omega_{mA})
\]
By rescaling the symplectic forms on both sides by $\frac{1}{m}$ we obtain the result. \end{proof}

In general it is possible that $\Omega \subseteq \Delta(Y,A,Y_\bullet)$ but $m\Omega \not\subseteq \Delta_m(Y,A,Y_\bullet)$ for any $m$. In the case of surfaces we have greater control from work of K\"uronya--Lozovanu \cite{kl_loc_18}. We begin with the following definition.

\begin{definition} \label{def:a_generic}
Fix a pseudo-polarized surface $(Y,A)$. We say that a valuation $\nu\colon\C(Y)\to\Z^2$ is $A$\emph{-generic} if $\nu=\nu_{Y_\bullet}$ for a flag $Y_\bullet$ on a birational model $\pi\colon\wt{Y}\to Y$ such that $\pi^*A\cdot Y_1>0$. We also call the Newton--Okounkov body $\Delta(Y,A,\nu)$ $A$\emph{-generic} when $\nu$ is $A$-generic.
\end{definition}

\begin{example} \label{ex:a_generic} If $Y_\bullet$ is a flag on $Y$ and $A$ is ample then $\nu_{Y_\bullet}$ is automatically $A$-generic.
\end{example}

\begin{lemma} \label{lem:a_generic}
Let $(Y,A)$ be a pseudo-polarized smooth surface and suppose that $\nu=\nu_{Y_\bullet}$ is $A$-generic. Then $\Delta(Y,A,\nu)$ is a convex moment domain.
\end{lemma}

\begin{proof}
Since $\Delta=\Delta(Y,A,\nu)$ is a convex polygon we only require that it has two adjacent edges on the coordinate axes. Observe that as $A$ is big and nef we have $N_\eps=0$ for all sufficiently small $\eps\geq0$ and so the lower boundary of $\Delta$ contains the line segment between $(0,0)$ and $(\eps,0)$ for some $\eps>0$. The upper bound for $\Delta$ at $t=0$ is given by $\beta_{Y_\bullet}(A,0)=A\cdot Y_1>0$ by assumption and thus the line segment between $(0,0)$ and $(0,A\cdot Y_1)$ is also contained in $\partial\Delta$.
\end{proof}

\begin{lemma}
Let $(Y,A)$ be a pseudo-polarized surface. If $\nu$ is $A$-generic then $\Delta_m(Y,A,\nu)$ is a convex moment domain for all sufficiently large $m\in\Z_{\geq0}$.
\end{lemma}

\begin{proof}
Following \cite{kl_loc_18} we call a rational point in $\Delta(Y,A,\nu)$ \emph{valuative} if it is of the form $\frac{1}{m}\nu(s)$ for some $s\in H^0(mA)$, hence is also a rational point in $\Delta_m(Y,A,\nu)$. Since $\nu$ is $A$-generic it follows that a small right triangle $\Lambda$ with corner at the origin is contained in $\Delta(Y,A,\nu)$. \cite[Cor.~B(ii)]{kl_loc_18} implies then that any rational point in $
\Lambda$ lying on the coordinate axes is valuative and hence lies in some $\Delta_m(Y,A,\nu)$. Choose two rational points of the form $(a,0)$ and $(0,b)$; these both live in $\Delta_{m_\star}(Y,A,\nu)$ for some $m_\star\in\Z_{\geq0}$. Since the origin lies in $\Delta_{m_\star}(Y,A,\nu)$ when $A$ is nef, taking the convex hull we find that $\Delta_{m_\star}(Y,A,\nu)$ is a convex domain and therefore $\Delta_m(Y,A,\nu)$ is a convex domain for all $m\geq m_\star$.
\end{proof}

We sum up the relevant conclusions from this section in the following corollary.

\begin{cor} \label{cor:kaveh_embedding_domain_2d} Let $(Y,A)$ be a smooth polarized surface. Suppose $Y_\bullet$ is a locally smooth flag on $Y$ with associated Newton--Okounkov body $\Delta=\Delta(Y,A,Y_\bullet)$. If a moment domain $\Omega$ satisfies $\Omega \subseteq \on{int}(\Delta)$ then there there is a symplectic embedding $X_\Omega \to (Y,\omega_A)$. 
\end{cor}

\begin{proof}
This follows since in the situation of Kaveh embeddings we only consider valuations from flags on $Y$, which are $A$-generic by Ex.~\ref{ex:a_generic} when $A$ is ample.
\end{proof}

\section{Algebraic capacities} In this section we apply the tools developed in \S \ref{sec:alg}-\ref{sec:symp_emb} to prove the main results of this paper on algebraic capacities and symplectic embeddings. 

\subsection{Definition and properties} We begin by recalling the definition of the algebraic capacities of a (weakly) polarized surface. These are an algebraic incarnation of ECH capacities.

\begin{definition}[\!{\cite[Def.~1]{wor_alg_22}}] \label{def:alg_cap} The $k$th \emph{algebraic capacity} $\calg_k(Y,A)$ of a weakly polarized surface $(Y,A)$ is given by
\begin{equation} \label{eqn:alg_cap} \tag{$\spadesuit$}
    \calg_k(Y,A):=\inf_{D\in\on{Nef}(Y)_\Z}\{A\cdot D:\chi(\mO_Y(D))\geq k+\chi(\mO_Y)\}
\end{equation} \end{definition}
\noindent Definition \ref{def:alg_cap} is a kind of isoperimetric problem on the nef cone of $Y$. The pairing $A\cdot D$ is an \emph{area} quantity -- when $A$ is ample it is explicitly the symplectic area with respect to $\omega_A$ -- and $\chi(\mO_Y(D))$ is a \emph{volume} or \emph{index} quantity measuring in some sense the moduli of the divisor $D$. 

\vspace{3pt}

We will require several features of these invariants, including simplifications in sufficiently nice cases. These features are all found in \cite[\S2]{wor_alg_22} and apply whenever $A$ is big and nef.

\begin{lemma}[Minimizer] There is a nef $\Z$-divisor realising the infimum (\ref{eqn:alg_cap}).
\end{lemma}

\begin{lemma}[Alternative Formulas] Let $(Y,A)$ be a pseudo-polarized surface. Then we have the following alternative formulas for the algebraic capacities.
\begin{itemize}
    \item (Toric) If $Y$ is toric, then
    \[\calg_k(Y,A) = \chom_k(Y,A) := \inf_{D\in\on{Nef}(Y)_\Z}\{A\cdot D:h^0(D)\geq k+1\}\]
    \item (Pseudo-Effective) If $Y$ is smooth or toric, $\calg$ can be computed with pseudo-effective divisors
    $$\calg_k(Y,A)=\inf_{D\in\neb(Y)_\Z}\{A\cdot D:\chi(\mO_Y(D))\geq k+\chi(\mO_Y)\}$$
    \item (Effective Anticanonical) If $-K_Y$ is effective, then it is equivalent to optimise over pseudo-effective $\Q$- or $\R$-divisors, and moreover in this case
    \[\calg_k(Y,A) \ge \chom_k(Y,A)\]
    \item (Index) If $Y$ is smooth or has only canonical singularities then
    $$\calg_k(Y,A)=\inf_{D\in\on{Nef}(Y)_\Z}\{A\cdot D:I(D)\geq 2k\}$$
    where $I(D)$ denotes the \emph{index} of the divisor $D$ (c.f.~\cite[Prop.~4,3]{cri_sym_19}) given by the formula \[I(D) := D\cdot(D-K_Y) = 2(\chi(\mO_Y(D))-\chi(\mO_Y))\]
\end{itemize}
\end{lemma}

\noindent Finally, the following result provides a connection linking algebraic capacities to ECH capacities.

\begin{thm}[\!{\cite[Thm.~1.1]{wor_ech_21} + \cite[Thm.~1.5]{cw_ech_20}}] \label{thm:cw_ech_vs_alg}
Let $(Y,A)$ be a polarized rational surface that is either smooth or toric. If $(X,\omega)$ is a star-shaped domain that symplectically embeds into $(Y,\omega_A)$ then
$$\cech_k(X,\omega)\leq\calg_k(Y,A)$$
Moreover, if $(Y,A)$ is toric with moment polytope $\Omega\subseteq\R_{\geq0}^2$ then
$$\cech_k(X_\Omega)=\calg_k(Y,A)$$
\end{thm}

\subsection{Weight sequences} \label{sec:weight_seq} We next introduce an algebraic analogue of weight sequences.

\begin{definition}
A \emph{tower of (pseudo-)polarized rational surfaces} $(\mathcal{Y},\mathcal{A})$ is a sequence 
\begin{equation} \label{eqn:tower} \tag{$\star$}
    Y_n\overset{\pi_n}{\longrightarrow} Y_{n-1}\overset{\pi_{n-1}}{\longrightarrow} \dots \overset{\pi_2}{\longrightarrow} Y_{1}\overset{\pi_1}{\longrightarrow} Y_0=\pr^2
\end{equation}
of one-point blowups such that each surface $Y_i$ comes with an ample (or big and nef) $\R$-divisor $A_i$ satisfying
$$A_0=\mO_{\pr^2}(c)\qquad A_i=\pi_i^*A_{i-1}-a_iE_i$$
Here $E_i$ is the exceptional fibre of the blowup $\pi_i$. The \emph{weight sequence} of $(\mathcal{Y},\mathcal{A})$ is
\[\on{wt}(\mathcal{Y},\mathcal{A}):=(c;a_1,\dots,a_n)\]
We say that a pseudo-polarized rational surface $(Y,A)$ \emph{admits} the weight sequence $w = (c;a_1,\dots,a_n)$ if there is a tower  $(\mathcal{Y},\mathcal{A})$ with $\on{wt}(\mathcal{Y},\mathcal{A}) = w$ and a birational morphism $\ph\colon Y_n\to Y$ with $\ph^*A=A_n$.
\end{definition}

\begin{remark} We note that this sense of weight sequence has been used or observed in several previous works including \cite{wor_alg_22,wor_tow_22,chmp_inf_20}. \end{remark}

Weight sequences are particularly useful when they satisfy some additional hypotheses.

\begin{definition} \label{def:pos_polytopal_ws} The weight sequence $w = (c;a_1,\dots,a_n)$ of a tower $(\mathcal{Y},\mathcal{A})$ is called
\begin{itemize}
    \item \emph{positive} if there is a (possibly different) tower $(\mathcal{Y}',\mathcal{A}')$ with weight sequence $w$ such that each of the surfaces $Y_i'$ in the tower has effective anticanonical divisor,
    \item \emph{polytopal} if there is a (possibly different) tower $(\mathcal{Y}',\mathcal{A}')$ with weight sequence $w$ such that each surface $Y_i$, each blowup $\pi_i$, and each divisor $A_i$ is toric.
\end{itemize}
Equivalently, a weight sequence $w$ is polytopal if it is the weight sequence of a convex, rational-sloped moment domain (see Def.~\ref{def:cvx_weight_seq}). Note that polytopal weight sequences are positive.
\end{definition}

The main result of this section is that, in the positive case, algebraic capacities are determined by the weight sequence. This is analogous to similar results for ECH capacities \cite[Thm.~A.1]{cri_sym_19}. 

\begin{prop} \label{prop:bounds}
Let $(Y,A)$ and $(Y',A')$ be two pseudo-polarized rational surfaces with the same positive weight sequence $w = (c;a_1,\dots,a_n)$. Then the algebraic capacities of $(Y,A)$ and $(Y',A')$ coincide, i.e.
\[\calg_k(Y,A)=\calg_k(Y',A') \qquad\text{for all}\qquad k \in \N\]
\end{prop}

We will make significant use of the fact that $\calg_k(Y,A)$ can be computed using nef divisors or effective divisors \cite[Prop.~2.11]{wor_alg_22}.

\begin{proof} First, assume that $(Y',A') = (Y_n',A_n')$ for some tower $(\mathcal{Y}',\mathcal{A}')$ where each surface $Y'_i$ has effective anticanonical divisor. 

\vspace{3pt}

Begin by noting that a tower as in (\ref{eqn:tower}) induces an identification $\on{Pic}(Y)_\R\cong\R^\rho$ where $\rho$ is the Picard rank of $Y$ as follows. Let $p_i=\pi_n\circ\pi_{n-1}\circ\dots\circ\pi_{i+1}$ and select the natural basis $\{H,B_1,\dots,B_n\}$ of $\on{Pic}(Y)_\R$ given by
\[H = B_0 = p_0^*\mO_{\pr^2}(1) \qquad \text{and}\qquad B_i = p_i^*E_i\]
In this basis, the intersection pairing $Q_Y$ on $\on{Pic}(Y)_\R$ is the diagonal bilinear form
\[\left[\begin{array}{cc}
1 & 0\\
0 & -\on{Id}_n
\end{array}\right] \]
We have an analogous basis $\{H',B_1',\dots,B_n'\}$ of $\on{Pic}(Y')_\R \simeq \R^\rho$ and thus an isomorphism
\[\Phi:(\on{Pic}(Y)_\R,Q_Y) \simeq (\on{Pic}(Y')_\R,Q_Y') \qquad\text{such that}\qquad \Phi(K_Y) = K_{Y'}\]
We will mostly use the notation $D' = \Phi(D)$. Note that the index of $D$ and $D'$ agree since
\[I(D)= D\cdot(D-K_Y) = D'\cdot(D'-K_{Y'}) = I(D')\]
Also note that if a divisor class $D$ has positive index then $D \cdot H > 0$ (or equivalently $D' \cdot H' > 0$).

\vspace{3pt}

We next claim that $\Phi$ and $\Phi^{-1}$ map nef divisors of positive index to effective divisors. That is
\begin{equation} \label{eqn:Phi_nef_to_eff}
\Phi(\{D \in \on{Nef}(Y)_\Z \; : \; I(D) > 0\}) \subseteq \neb(Y')_\Z
\end{equation}
\begin{equation} \label{eqn:Phi_inv_nef_to_eff}
\Phi^{-1}(\{D' \in \on{Nef}(Y')_\Z \; : \; I(D') > 0\}) \subseteq \neb(Y)_\Z
\end{equation}
We first prove (\ref{eqn:Phi_nef_to_eff}). Choose $D \in\on{Nef}(Y)_\Z$. Since $I(D) = I(D') > 0$, we have
\[0 > -D \cdot H = -D' \cdot H'\]
Since $H'$ is nef, $-D'$ cannot be effective. Therefore
\[h^2(D') = h^0(K_Y' - D') = 0\]
It then follows from $I(D') \geq0$ that $h^0(D')\geq 1$ and thus that $D' \in\neb(Y')_\Z$. To see (\ref{eqn:Phi_inv_nef_to_eff}), note that if $D$ is an ample $\Z$-divisor on $Y$ then
\[I(mD)=mD\cdot(mD-K_Y)=m^2D^2-mD\cdot K_Y>0\quad\text{for all $m\gg0$}\]
In particular, the ample cone over $\R$ can be written as
\[\on{Amp}(Y)_\R = \on{cone}\{D\in\on{Amp}(Y)_\Z:I(D)\geq0\}\]
Note that for a subset $S\subseteq \on{Div}(Y)_\Z$,  we have a well-defined (closed) dual cone
$$S^\vee = \on{cone}(S)^\vee :=\{D\in\on{Div}(Y)_\Z:D \cdot S \geq0\text{ for all $S\in S$}\}$$
Since the closure of the ample cone is the nef cone, and the effective cone and the nef cone are dual, we have
\begin{equation} \label{eqn:Phi_inv_nef_to_eff_2}
\{D \in\on{Amp}(Y)_\Z\; : \; I(D)\geq0\}^\vee=\on{Amp}(Y)^\vee_\Z = \neb(Y)_\Z \end{equation}
On the otherhand, we know by (\ref{eqn:Phi_nef_to_eff}) that
\begin{equation} \label{eqn:Phi_inv_nef_to_eff_1}\{D \in\on{Amp}(Y)_\Z\; : \; I(D)\geq0\} \subseteq \Phi^{-1}(\neb(Y')_\Z)\end{equation}
By dualizing (\ref{eqn:Phi_inv_nef_to_eff_1}) and applying (\ref{eqn:Phi_inv_nef_to_eff_2}), we acquire (\ref{eqn:Phi_inv_nef_to_eff}) as follows.
\[\neb(Y)_\R \supseteq \Phi^{-1}(\neb(Y')_\R^\vee) \supseteq \Phi^{-1}(\{D \in \on{Nef}(Y')_\Z \; : \; I(D) > 0\})\]
Finally, we show that the algebraic capacities of $(Y',A')$ and $(Y,A)$ bound each other. Simply apply (\ref{eqn:Phi_nef_to_eff}) to see that
\[\calg_k(Y',A') = \inf_{D'\in\neb(Y')_\Z}\{A'\cdot D':I(D')\geq 2k\} \ge \inf_{D\in\on{Nef}(Y)_\Z}\{A\cdot D:I(D)\geq 2k\} = \calg_k(Y,A) \]
The opposite direction is an identical use of (\ref{eqn:Phi_inv_nef_to_eff}). This concludes the proof in the case where $(Y',A') = (Y_n',A_n')$ for a tower $(\mathcal{Y}',\mathcal{A}')$ where each surface $Y'_i$ has effective anticanonical divisor
\vspace{3pt}

In the general case, by Def.~\ref{def:pos_polytopal_ws}, there is a tower $(\mathcal{Y}'',\mathcal{A}'')$ with weight sequence $w$ such that each surface $Y''_i$ in the tower has effective anticanonical divisor. Thus by the special case, we have 
\[
\calg_k(Y,A) = \calg_k(Y''_n,A_n'') = \calg_k(Y'_n,A_n') \qquad\text{for each }k \qedhere
\]\end{proof}

\begin{remark} The proof of Prop.~\ref{prop:bounds} used the mutual identification of the Picard groups of $Y$ and $Y'$ to $\R^\rho$ with the given bilinear form to translate between divisor classes on $Y$ and $Y'$, and to convert positivity properties on $Y$ (nefness, index sufficiently large) to positivity properties on $Y'$ using appropriate vanishing. This implies that a tower with effective $-K_Y$ is the most efficient way to blow-up, in the sense that the effective cone is maximised. We expect that towers where the centres are in very general position will have `minimal' effective cone in this sense.
\end{remark}

\begin{remark} The proof of Prop.~\ref{prop:bounds} simplifies if $-K_Y$ is effective, since the use of (\ref{eqn:Phi_nef_to_eff}) also works in place of (\ref{eqn:Phi_inv_nef_to_eff}).
\end{remark}

In practice, it is often more convenient to show that a weight sequence is polytopal than positive as this is demonstrable through purely combinatorial arguments and, as we shall see, plays well with the theory of Newton--Okounkov bodies.

\begin{example}
Note that any weight sequence of the form $(c,\eps_1,\dots,\eps_n)$ with $\eps_i>0$ small relative to $c$ is polytopal and hence positive. One can thus compute the algebraic capacities of any pseudo-polarized rational surface coming from a tower with this weight sequence by the much easier task of computing the algebraic capacities for a toric surface with this weight sequence. We will develop this further in \S\ref{sec:higher_blowups}.
\end{example}

\subsection{Calculation via Newton-Okounkov bodies} \label{sec:calc_via_NO_bodies} We now use Newton-Okounkov bodies to compute algebraic capacities and prove general embedding results. 

\begin{thm} \label{thm:alg_ech}
Let $(Y,A)$ be a pseudo-polarized smooth rational surface. Suppose $(Y,A)$ has a Newton--Okounkov body $\Delta$ whose weight sequence is admitted by $(Y,A)$.
Then, the ECH capacities of $F_\Delta$ and the algebraic capacities of $(Y,A)$ agree:
\[ \calg_k(Y,A) = \cech_k(F_\Delta)\]
If $\Delta$ is $A$-generic then
\[ \calg_k(Y,A) = \cech_k(X_\Delta) \]
\end{thm}

\begin{proof} By assumption the weight sequence of $(\mathcal{Y},\mathcal{A})$ is polytopal and so positive. Therefore we can apply Prop.~\ref{prop:bounds} to see that
\[\calg_k(Y,A) = \calg_k(Y_\Delta,A_\Delta)\]
The algebraic capacities of a toric surface agree with the ECH capacities of the corresponding free toric domain \cite[Thm.~1.1]{wor_ech_21} and with the toric domain $X_\Delta$ if $\Delta$ is $A$-generic and hence a moment domain.
\end{proof}

As an application of Thm.~\ref{thm:alg_ech} we have the following sharp embedding result.

\begin{cor} \label{cor:conc_emb} Let $(Y,A)$ be a polarized smooth rational surface. Suppose there is a Newton--Okounkov body $\Delta(Y,A,Y_\bullet)$ coming from a flag $Y_\bullet$ on $Y$ such that $(Y,A)$ admits the weight sequence $\on{wt}(\Delta)$. Then for any concave moment domain $\Omega$ the following are equivalent:
\begin{itemize}
    \item[(a)] $\cech_k(X_\Omega) \le \calg_k(Y,A)$ for all $k \in \N$.
    \vspace{3pt}
    \item[(b)] There is a symplectic embedding $X_{c\Omega} \to (Y,\omega_A)$ for any $0 < c < 1$.
\end{itemize}
\end{cor}

\begin{proof} First, (b) implies (a) by Thm.~\ref{thm:cw_ech_vs_alg}. To show that (a) implies (b), let $\Delta := \Delta(Y,A,Y_\bullet)$ be the Newton--Okounkov body $(Y,A)$. By Thm.~\ref{thm:alg_ech} we know that
\[\cech_k(X_\Delta) = \cech_k(F_\Delta) = \calg_k(Y,A) \ge \cech_k(X_\Omega)\]
Then by Cor.~\ref{cor:kaveh_embedding_domain_2d} and Prop.~\ref{prop:cg_embedding} we have for any $b,c$ with $0 < c < b < 1$ there is a symplectic embedding
\[
X_{c\Delta} \xrightarrow{\text{Prop. \ref{prop:cg_embedding}}} X_{b\Omega} \xrightarrow{\text{Cor. \ref{cor:kaveh_embedding_domain_2d}}} (Y,\omega_A) \qedhere
\]\end{proof}

\subsection{Asymptotics via Zariski decomposition} \label{sec:zariski_methods} We next apply Zariski decomposition to study the asymptotics of ECH capacities.

\vspace{3pt}

In \cite[Thm.~4.2]{wor_alg_22} it was shown that for a big and nef divisor $A$ on a smooth or toric projective surface $Y$ we have
$$\lim_{k\to\infty}\frac{\calg_k(Y,A)^2}{k}=2A^2$$
This matches the asymptotics found by Cristofaro-Gardiner--Hutchings--Ramos \cite{chr_asy_15} for ECH capacities. We generalise this to cover the case where $A$ is a big divisor.

\begin{prop} \label{prop:weyl_big} Suppose $(Y,A)$ is a smooth or toric weakly polarized surface. Then
$$\lim_{k\to\infty}\frac{\calg_k(Y,A)^2}{k}=2\on{vol}(A)$$
\end{prop}

\begin{proof} As in \cite[Prop.~4.19]{wor_alg_22}, we reduce the toric case to the smooth case using a limiting argument. Let $A=P+N$ be the Zariski decomposition of $A$ and suppose that $P$ is a $\rz$-divisor
\[P=qP_0 \qquad\text{with}\qquad P_0 \in \on{Div}(Y)_\Z \text{ and }q \in \R_{>0}\]
We observe that $I(d_kP_0) = d_kP_0\cdot(d_kP_0-K_Y)$ is bounded below by $2k$ when
\[d_k\geq\frac{1}{2P_0^2}\left(P_0\cdot K_Y+\sqrt{(P_0\cdot K_Y)^2+8kP_0^2}\right)\]
Thus for $d_k$ satisfying this bound, the algebraic capacities satisfy
\[\frac{\calg_k(Y,A)^2}{k}\leq\frac{\left(\lceil d_k\rceil A\cdot P_0\right)^2}{k}=\frac{\left(\lceil d_k\rceil qP_0^2\right)^2}{k}\]
Taking the limit as $k \to \infty$ and applying Lem.~\ref{lem:volume_zariski}, we find that
\[\lim_{k\to\infty}\frac{\calg_k(Y,A)^2}{k}\leq\lim_{k\to\infty}\frac{2kq^2P_0^2}{k}=2q^2P_0^2=2P^2=2\on{vol}(A)\]
Continuity implies that the same result holds for any big $\R$-divisor $A$ -- note this indeed follows since if $A$ is a $\Q$-divisor then so are $P$ and $N$ by Prop.~\ref{prop:zariski_decomposition}. Hence if $A$ is an $\rz$-divisor then so are $P$ and $N$. 

\vspace{3pt}

For the converse inequality we observe that for any nef divisor $D$ we have $A\cdot D\geq P\cdot D$ since $N$ is pseudo-effective. Hence
\[\calg_k(Y,A)\geq\calg_k(Y,P) \quad\text{and thus}\quad \lim_{k\to\infty}\frac{\calg_k(Y,A)^2}{k}\geq\lim_{k\to\infty}\frac{\calg_k(Y,P)^2}{k}=2P^2=2\on{vol}(A)\]
since $P$ is big and nef and so \cite[Thm.~4.2]{wor_alg_22} applies to $(Y,P)$.
\end{proof}

\begin{remark} The main philosophical takeaway is that for large $k$, the optimisers computing the $k$th algebraic capacity are close to multiples of the positive part $P$ of $A$. Thus we can approximate the $k$th algebraic capacity with multiples of $P$. \end{remark}

\begin{example} Consider the big non-nef divisor $A=H+E$ on $Y=\on{dP}_8$. $H$ is the positive part of $A$ and
$$\calg_k(Y,A)=\calg_k(Y,H)=\calg_k(\pr^2,\mO(1))$$
\end{example}

We also recreate results from \cite[\S3]{wor_alg_22} for big divisors to study their sub-leading asymptotics.

\begin{lemma} \label{lem:optimal} Suppose $A$ is a big $\Z$-divisor on a smooth surface $Y$ and let $A=P+N$ be its Zariski decomposition. Suppose that $D$ is an optimizer for the algebraic capacities in that
\[D \cdot A = \calg_k(Y,A) \qquad\text{where} \qquad k = I(D)\]
Then if $I(D)$ is sufficiently large, $D+P$ is also an optimizer.
\[(D + P) \cdot A = \calg_l(Y,A) \qquad\text{where} \qquad l = I(D+P)\].
\end{lemma}
\begin{proof} The argument exactly follows \cite[Prop.~4.5]{wor_alg_22} where the only extra ingredient required is the fact that $P\cdot D\to\infty$ as $I(D) \to\infty$, which followed automatically when $P=A$. This follows in general since $P\cdot D\geq\calg_k(Y,P)\to\infty$ as $k = I(D) \to\infty$.
\end{proof}

\begin{cor} Let $A$ be a big $\rz$-divisor on $Y$ with Zariski decomposition $A=P+N$. Let $P=qP_0$ for $P_0$ a primitive $\Z$-divisor. Then there are $k_0\in\Z_{\geq0}$ and $s_1<\dots<s_r\in\R$ such that $s_r-s_1<qP_0^2$ and
$$\{\calg_k(Y,A):k\geq k_0\}=\{s_i+mqP_0^2:m\geq0,i=1,\dots,r\}$$
\end{cor}

\begin{proof}
This follows the proof of \cite[Cor.~4.6]{wor_alg_22} with Lem.~\ref{lem:optimal} playing the role of \cite[Prop.~4.5]{wor_alg_22}.
\end{proof}

It also follows that optimisers for $\calg_k(Y,A)$ for large enough $k$ take the form $D_i+mP_0$ for some fixed set $\{D_1,\dots, D_r\}$ of effective divisors and $m\geq0$. We can conclude from this that the sub-leading asymptotics of $\calg_k(Y,A)$ are $O(1)$ but nonconvergent when $A$ is a big $\rz$-divisor, matching the conclusion in the big and nef case \cite[Thm.~4.10]{wor_alg_22}. The proof is identical given our setup.

\begin{cor} \label{cor:sublead}
Let $(Y,A)$ be a weakly polarized surface that is either smooth or toric. If $A$ is an $\rz$-divisor then the error terms
$$\ealg_k(Y,A):=\calg_k(Y,A)-\sqrt{2\on{vol}(A)k}$$
are $O(1)$ and nonconvergent with
$$\liminf_{k\to\infty}\ealg_k(Y,A)=\frac{1}{2}K_Y\cdot A\text{ and }\limsup_{k\to\infty}\ealg_k(Y,A)=\frac{1}{2}K_Y\cdot A+\on{gap}(Y,A)$$
where $\on{gap}(Y,A)>0$ is defined in \cite[Def.~4.8]{wor_alg_22}.
\end{cor}

\subsection{Homological formula via Zariski decomposition} We end this section by applying the Zariski decompoisition to slightly generalising a reformulation of algebraic capacities using pseudo-effective $\R$-divisors appearing in \cite{cw_ech_20,wor_alg_22}. We do not require this result elsewhere in the paper.

\vspace{3pt}

\begin{prop} \label{prop:eff_nef} If $(Y,A)$ is a pseudo-polarized $\Q$-factorial surface then
$$\inf_{D\in\NE(Y)_\R}\{A\cdot D:h^0(D)\geq k+1\}=\inf_{D\in\on{Nef}(Y)_\Z}\{A\cdot D:h^0(D)\geq k+1\}$$
\end{prop}

Prop.~\ref{prop:eff_nef} shows that the homological version of the algebraic capacities can be evaluated by considering either nef $\Z$-divisors or pseudo-effective $\R$-divisors. In the toric case or the weak del Pezzo case \cite[Cor.~2.5]{wor_alg_22}, we can conclude the same for algebraic capacities.

\begin{proof} Since both sides are invariant under blowup and pullback, we may assume that $Y$ is smooth. The $\leq$ inequality is clear. Let $D$ be a pseudo-effective $\R$-divisor and let $D=P+N$ be its Zariski decomposition. Since $N$ is pseudo-effective we have $A\cdot D\geq A\cdot P$. Since $h^0(P)=h^0(D)\geq k+1$ we have
$$A\cdot D\geq A\cdot P\geq\inf_{D\in\on{Nef}(Y)_\R}\{A\cdot D:h^0(D)\geq k+1\}$$
and so
$$\inf_{D\in\NE(Y)_{\rz}}\{A\cdot D:h^0(D)\geq k+1\}\geq\inf_{D\in\on{Nef}(Y)_\R}\{A\cdot D:h^0(D)\geq k+1\}$$
Let $P$ be a nef $\R$-divisor. Note that
$$A\cdot\lf P\rf\leq A\cdot P\qquad\text{and}\qquad h^0(\lf P\rf)=h^0(P)$$
and so $\lf P\rf$ is a candidate $\Z$-divisor to evidence the result of the proposition. However, $\lf P\rf$ may not be nef. The solution is to take the Zariski decomposition of $\lf P\rf$ and iterate. More precisely, let $P_0$ be a nef $\R$-divisor and let the Zariski decomposition of $\lf P_0\rf$ be $P_1+N_1$. If $N_1=0$ then $\lf P_0\rf$ is nef and we are done. If not, continue by setting
$$\lf P_i\rf=P_{i+1}+N_{i+1}$$
Note that
$$A\cdot P_i\geq A\cdot\lf P_i\rf\geq A\cdot P_{i+1}\qquad\text{and}\qquad h^0(P_i)=h^0(\lf P_i\rf)=h^0(P_{i+1})$$
so that $P_{i+1}$ is preferable to $P_i$ from the perspective of the optimisation problems in the proposition. Let $P_i=\sum_pa_p^iD_p$ so that $a_p^{i+1}\leq\lf a_p^i\rf$. By property 3.~of Zariski decomposition in \S\ref{sec:zariski} each $a_p^i$ is a decreasing sequence that is bounded below (by zero) and thus converges. By construction the limit must be an integer and so the divisors $P_i$ converge to a $\Z$-divisor. Since the nef cone is closed this limiting divisor is a nef $\Z$-divisor and the result is proven.
\end{proof}

\begin{remark} Note that $h^0(mP_i)\not= h^0(mP_{i+1})$ for $m>1$ in general due to taking the round-down. Similar results were used in \cite{cw_ech_20,wor_alg_22,cw_lat_21} to relate ECH capacities to algebraic capacities by connecting Seiberg--Witten or Gromov--Witten invariants to nefness. We anticipate that arguments of this type will be useful in relating other symplectic invariants to birational qualities of divisors. \end{remark}

\section{Examples and applications} \label{sec:app} In this section, we provide several families of examples that our main results on algebraic capacities (Thm.~\ref{thm:alg_ech}) and symplectic embeddings (Cor.~\ref{cor:conc_emb}) apply to. In particular, we compute Newton--Okounkov bodies with suitable weight sequences for these examples.

\begin{remark} All Newton--Okounkov bodies considered from this point will be automatically $A$-generic and so we omit this from the discussion for brevity. \end{remark}

\subsection{del Pezzo 5} \label{sec:dp5} As a first example we consider the del Pezzo surface $Y$ of degree $5$ polarized by its anticanonical divisor $A=-K_Y$. This the highest degree non-toric del Pezzo surface, obtained by blowing up $\pr^2$ in four general points. Consider the flag $Y_\bullet = \{y\in Y_1\subseteq Y\}$ where $Y_1$ is a $(-1)$-curve and $y$ is a general point on $Y_1$. 

\begin{claim} \label{cl:NO_of_dP5} The Newton--Okounkov body $\Delta=\Delta(Y,A,Y_\bullet)$ is as depicted in Fig.~\ref{fig:no_dp5}.
\end{claim}

\noindent This has weight sequence $(3;1,1,1,1)$ as computed in Fig.~\ref{fig:wt_dp5}, which coincides with the algebraic weight sequence of the natural tower presentation of $(Y,A)$. Therefore Cor.~\ref{cor:conc_emb} implies that the algebraic capacities of $(Y,-K_Y)$ sharply obstruct embeddings of concave toric domains into $(Y,A)$.

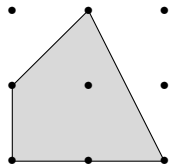
\begin{figure}[h]
    \caption{Newton--Okounkov body for del Pezzo $5$}
    \label{fig:no_dp5}
    \centering
    \begin{tikzpicture}
    \draw[fill=gray!30] (0,0) to (0,1) to (1,2) to (2,0) to (0,0);
    \foreach \i in {0,1,2}
    {
    \foreach \j in {0,1,2}
    \node (\i\j) at (\i,\j){\tiny$\bullet$};
    }
    \end{tikzpicture}
\end{figure}

\begin{figure}[h]
    \caption{Weight sequence decomposition for $\Delta$}
    \label{fig:wt_dp5}
    \centering
    \begin{tikzpicture}
    
    \draw[fill=gray!60] (0,1) to (0,3) to (1,2) to (0,1);
    \draw[fill=gray!80] (1,2) to (3,0) to (2,0) to (1,2);
    \draw (2,0) to (3,0);
    \draw (0,1) to (0,3);
    
    \draw[fill=gray!30] (0,0) to (0,1) to (1,2) to (2,0) to (0,0);
    
    \draw[fill=gray!60] (5,0) to (5,1) to (7,0) to (5,0);
    
    \draw[fill=gray!80] (9,0) to (9,2) to (10,0) to (9,0);
    
    \draw[dashed, color=white, line width=0.8pt] (5,1) to (6,0);
    
    \draw[dashed, color=white, line width=0.8pt] (9,1) to (10,0);
    
    \foreach \i in {0,1,2,3}
    {
    \foreach \j in {0,1,2,3}
    \node (\i\j) at (\i,\j){\tiny$\bullet$};
    }
    
    \foreach \i in {0,1,2}
    {
    \foreach \j in {0,1,2}
    {
    \node (\i\j) at (5+\i,\j){\tiny$\bullet$};
    \node (\i\j) at (9+\i,\j){\tiny$\bullet$};
    }
    }
    \end{tikzpicture}
\end{figure}
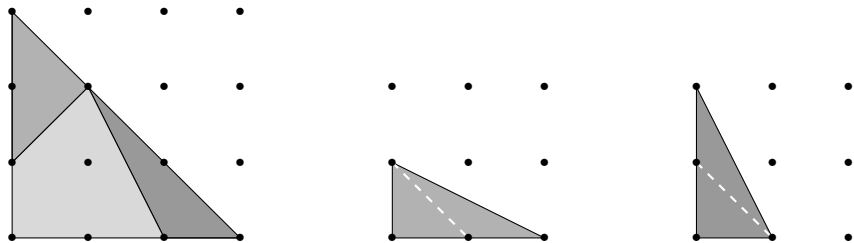

\begin{proof}[Proof of Claim \ref{cl:NO_of_dP5}] We compute $\Delta$ explicitly by analyzing Zariski chamber crossings. Note that the effective cone of a del Pezzo surface is generated by its $(-1)$-curves hence we need only test against these curves to verify nefness. We follow the notation of \cite[\S3.1]{duc_the_21} for $(-1)$-curves on $Y$ as shown in Fig.~\ref{fig:dp_pic}. In particular, we denote by $E_1,E_2,D_3,D_5$ the exceptional curves on $Y$. There are then two $5$-cycles of $(-1)$-curves $D_1+\dots+D_5$ and $E_1+\dots+E_5$, each in the class $-K_Y$.

\begin{figure}[h]
    \caption{$(-1)$-curves on del Pezzo $5$}
    \label{fig:dp_pic}
    \centering
    \begin{tikzpicture}[scale=1.6]
    \draw (0,-0.5) to (0,2.5);
    \draw (-0.5,0) to (2.5,0);
    \draw (-0.5,2.5) to (2.5,-0.5);
    \draw (-0.5,1.5) to (1.5,-0.5);
    
    \draw (-0.5,1.25) to (2.5,-0.25);
    \draw (1.25,-0.5) to (-0.25,2.5);
    
    \node[fill=white] (a) at (0,2){\small$D_5$};
    \node[fill=white] (b) at (2,0){\small$D_3$};
    \node[fill=white] (c) at (0,1){\small$E_1$};
    \node[fill=white] (d) at (1,0){\small$E_2$};
    
    \node[fill=white] (d1) at (0,-0.5){\small$D_1$};
    \node[fill=white] (d2) at (-0.5,0){\small$D_2$};
    \node[fill=white] (d4) at (1,1){\small$D_4$};
    \node[fill=white] (e3) at (1.2,0.4){\small $E_3$};
    \node[fill=white] (e5) at (0.4,1.2){\small $E_5$};
    \node[fill=white] (e4) at (0.45,0.45){\small $E_4$};
    \end{tikzpicture}
\end{figure}
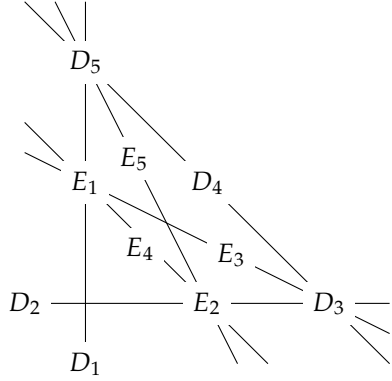

We denote by $\pi\colon Y\to \pr^2$ the blowup in the four points corresponding to $E_1,E_2,D_3,D_5$ in Fig.~\ref{fig:dp_pic}. Pick $Y_1=D_1$ for concreteness. The divisor $A_t$ whose journey through the Zariski chamber decomposition determines $\Delta$ is given in the basis $H=\pi^*\mO_{\pr^2}(1),E_1,E_2,D_3,D_5$ by
$$A_t=3H-E_1-E_2-D_3-D_5-tD_1=(3-t)H-(1-t)E_1-E_2-D_3-(1-t)D_5$$
This is nef for $0\leq t\leq 1$; indeed, it is apparent that $A_t\cdot E_1<0$ when $t>1$. Hence, $P_t=A_t$ and $N_t=0$ for $0\leq t\leq 1$. Set $t=1+s$. We can write
\begin{align*}
    A_t=P_t+N_t \quad\text{where}\quad P_t&=(2-2s)H-(1-s)E_2 -(1-s)D_3 \\
    N_t&=sH+sE_1-sE_2-sD_3+sD_5
\end{align*}
and can verify that $P_t$ is nef and $N_t$ is effective for $1\leq t\leq 2$ or equivalently $0\leq s\leq 1$. Indeed, $N_t=s(E_1+D_2+D_5)$ and $P_t\cdot H<0$ for $t>2$. Further, $P_t\cdot N_t=0$ giving that this is the Zariski decomposition for $A_t$ in this interval. As indicated by the intersection with $H$, for $t>2$ we find that $A_t$ is no longer big and so by \cite[Thm.~6.4]{lm_con_09} the functions $\alpha_{Y_\bullet}(A,\cdot)$ and $\beta_{Y_\bullet}(A,\cdot)$ carving out $\Delta$ are supported on $[0,2]$. In this case the lower boundary $\alpha_{Y_\bullet}(A,\cdot)\equiv0$ since $N_0=0$ and $N_t$ never acquires new components meeting $y\in D_1$. The top boundary is given by
$$\beta_{Y_\bullet}(A,t)=P_t\cdot D_1=\begin{cases}
1+t & 0\leq t\leq 1 \\
4-2t & 1\leq t\leq 2
\end{cases}$$
producing the polygon in Fig.~\ref{fig:no_dp5}.  \end{proof}

\subsection{del Pezzo surfaces} \label{sec:higher_dp} As a more sophisticated example we consider a polarized del Pezzo surface $(Y,A)$ presented as a tower $(\mathcal{Y},\mathcal{A})$ as in (\ref{eqn:tower}) with weight sequence $(c;a_1,\dots,a_n)$. We fix a flag $Y_\bullet$ of the following form:
\[Y_\bullet = y \in E_\star \subset Y\]
where $E_\star$ is an exceptional curve from some $\pi_\star$ on $Y$ and $y \in E_\star$ is an arbitrary point. Let $\Delta$ be the associated Newton--Okounkov body $\Delta=\Delta(Y,A,Y_\bullet)$.

\begin{prop} \label{prop:example} Let $r$ be the Picard rank of $Y$ and suppose that any of the following criteria holds.
\begin{enumerate}[label=(\arabic*)]
    \item $r \le 5$
    \vspace{3pt}
    \item $r \le 7$ and $c-\sum_{j=1}^4 a_{i_j}\geq 0$ for all distinct $i_1,\dots,i_4$.
    \vspace{3pt}
    \item $r \le 8$ and $c-\sum_{j=1}^6 a_{i_j}\geq0$ for all distinct $i_1,\dots,i_6$.
    \vspace{3pt}
    \item $r \le 9$ and in addition to the inequalities from (3) we have $3c-2\sum_{j=1}^7 a_{i_j} \geq0$ for all distinct $i_1,\dots,i_7$.
    \vspace{3pt}
\end{enumerate}
Then the Newton--Okounkov body $\Delta$ associated to the flag $Y_\bullet$ has weight sequence equal to the algebraic weight sequence of the tower ($\mathcal{Y},\mathcal{A})$.
\end{prop}

\noindent Each of the inequalities in (2)-(4) is quantifying that $c$ is appropriately large relative to $a_1,\dots,a_n$.

\begin{proof}
We proceed by induction on the Picard rank. For the base case, if $Y$ has Picard rank at most $4$ then $Y$ is toric, hence each $(-1)$-curve is torus-invariant and so by \cite[\S6.1]{lm_con_09} the Newton--Okounkov body $\Delta$ for the flag $Y_\bullet$ is the moment polytope of $(Y,A)$. In particular, it is polytopal.

\vspace{3pt}

In general, suppose the statement is true for polarized del Pezzo surfaces of Picard rank $n$ and suppose $(Y,A)$ is a polarized del Pezzo surface of Picard rank $n+1$. Without loss of generality assume that $E_\star$ is not the exceptional curve from the final blowup in the tower and consider the blowup map $\pi_n\colon Y\to Y_{n-1}$ from contracting the final exceptional curve $E_n$. Note that $Y_{n-1}$ is a del Pezzo surface of Picard rank $n$ and that $A=\pi_n^*A_{n-1}-a_n E_n$ for some ample divisor $A_{n-1}$ on $Y_{n-1}$ so that $(Y_{n-1},A_{n-1})$ is another polarized del Pezzo surface.

\vspace{3pt}

$E_\star$ can be viewed as a $(-1)$-curve in $Y_{n-1}$ and, further, the inequalities holding between $c,a_1,\dots,a_n$ imply the analogous inequalities hold among $c,a_1,\dots,a_{n-1}$. Thus the Newton--Okounkov body $\Delta(Y_{n-1},A_{n-1},\ol{Y}_\bullet)$ has the same weight sequence $(c;a_1,\dots,a_{n-1})$ as $(Y_{n-1},A_{n-1})$ where $\ol{Y}_\bullet$ is the flag $\ol{y}\in E_\star\subseteq\ol{Y}$.

\vspace{3pt}

We will show in each case that the Newton--Okounkov body associated to $A$ with respect to the flag $Y_\bullet:y\in E_0\subseteq Y$ is obtained from $\Delta(Y_{n-1},A_{n-1},\ol{Y}_\bullet)$ by slicing with a hyperplane in such a way that a \emph{regular triangle} (i.e. the $\on{SL}(2,\Z)$ image of a right, isosceles triangle) of height (and width) $a_n$ is removed. This will show that the weight sequence of $\Delta(Y,A,Y_\bullet)$ is obtained from the weight sequence of $\Delta(Y_{n-1},A_{n-1},\ol{Y}_\bullet)$ by concatenating with $\{a_n\}$, which is by construction how one obtains the weight sequence of $(Y,A)$ from the weight sequence of $(Y_{n-1},A_{n-1})$.

\vspace{3pt}

From Ex.~\ref{ex:-1_curves} we can explicitly describe the Zariski chamber decomposition for $\on{Big}(Y)$: in this case the Zariski chambers are carved out by the hyperplane arrangement
$$\{C^\perp\}_{C^2<0}$$
As in \S\ref{sec:chamber} we write $(t,h)$ for points of Newton--Okounkov bodies in $\R^2$. Recall that the functions defining the upper and lower boundaries of $\Delta(Y,A,Y_\bullet)$ have breaks exactly when $A_t=A-tE_\star$ crosses into a new Zariski chamber; in other words, when $A_t\cdot C=0$ for some $(-1)$-curve $C$ on $Y$. Write $p\colon Y\to\pr^2$ for the total blowup map, which factors through a map $\ol{p}\colon Y_{n-1}\to\pr^2$. Write $H=p^*\mO_{\pr^2}(1)$. We compare the $(-1)$-curves on $Y$ to those on $Y_{n-1}$. They split into four classes:
\begin{enumerate}
    \item $C=\pi_n^*\ol{C}$ for some $(-1)$-curve $\ol{C}\subseteq Y_{n-1}$,
    \item $C=E_n$,
    \item $C=H-E_i-E_n$,
    \item $C$ is the strict transform of a curve of degree $>1$ in $\pr^2$ passing through the point $p(E_n)$.
\end{enumerate}
Note that the centre of the blowup $\pi\colon Y\to Y_{n-1}$ cannot be on any $(-1)$-curve and so all $(-1)$-curves in $Y_{n-1}$ lift to $(-1)$-curves in $Y$. The main work of the proof is contained in the following.

\begin{claim} \label{claim:aux}
If the weight sequence of $(Y,A)$ satisfies one of the conditions (1)-(4), then the only $(-1)$-curves $C$ inducing a wall $C^\perp$ that $A_t$ meets as $t$ varies are those from (i) and $C=H-E_\star-E_n$.
\end{claim}

For clarity we write $\ol{A}=A_{n-1}$ and $\ol{A}_t=A_{n-1}-tE_\star$. Note for a curve $C=\pi^*\ol{C}$ in (i) we have
$$A_t\cdot C=(\pi_n^*\ol{A}-a_n E-tE_\star)\cdot\pi_n^*E_0=\pi_n^*(\ol{A}-tE_\star)\cdot\pi_n^*\ol{C}=\ol{A}_t\cdot \ol{C}$$
Hence any wall $\ol{C}_i^\perp$ that $\ol{A}_t$ hits at $t=t_i$ also defines a wall $C_i^\perp$ that $A_t$ hits at $t=t_i$. This proves the part of Claim \ref{claim:aux} regarding curves in (i).

\vspace{3pt}

We claim that any curve $C$ defining a wall that $A_t$ meets must be supported on $E_\star$ when written in the basis $H,E_1,\dots,E_n$ of $\on{Pic}(Y)$. Indeed, if $C$ has no $E_\star$ component then $A_t\cdot C=A\cdot C>0$. Thus, the remaining relevant curves fall into the subclasses:
\begin{enumerate}
    \item[(iii')] $C=H-E_\star-E_n$,
    \item[(iv')] $C$ is the strict transform of a curve of degree $>1$ in $\pr^2$ passing through the points $p(E)$ and $p(E_\star)$.
\end{enumerate}

The curve $C=H-E_\star-E_n$ in (iii') does describe a wall that $A_t$ passes through since
$$A_t\cdot C=(\pi^*\ol{A}-a_n E_n-tE_\star)\cdot(H-E_\star-E_n)=\ol{A}\cdot(H-E_\star)-a_n-t$$
and so $A_t$ meets this new wall at $t=t_\star:=\ol{A}\cdot(H-E_\star)-a_n$. Indeed, in this case we find that the upper bound $\mu$ for $t$ is given by $\ol{A}\cdot(H-E_\star)=A\cdot(H-E_\star)$ for both $\Delta(Y_{n-1},\ol{A},\ol{Y}_\bullet)$ and $\Delta(Y,A,Y_\bullet)$. This follows since $H-E_\star$ is a $0$-curve on $Y_{n-1}$ and $Y$, and when $A_t\cdot(H-E_\star)=0$ we see that $A_t$ has met the corresponding face of the effective cone. We see then that $t_\star=\mu-a_n$ and thus defines a new vertex of $\Delta(Y,A,Y_\bullet)$ compared to $\Delta(Y_{n-1},\ol{A},\ol{Y}_\bullet)$.

\vspace{3pt}

It remains to show that in each of the relevant cases all curves $C$ in (iv') define a wall that $A_t$ does not meet inside $\on{Big}(Y)$. We write
$$C=dH-\sum_{i=1}^nb_iE_i$$
where $d^2-\sum_{i=1}^nb_i^2=-1$ and $b_\star,b_n>0$. Compute
$$A_t\cdot C=A\cdot(dH-\sum_{i=1}^nb_iE_i)-b_\star t$$
so that if $A_t$ did meet $C^\perp$ it would be when
$$t=t_\diamond:=\frac{1}{b_\star}A\cdot(dH-\sum_{i=1}^nb_iE_i)=\frac{1}{b_\star}(cd-\sum_{i=1}^na_ib_i)$$
We show that $A_{t_\diamond}$ is not big, equivalently $t_\diamond\geq\mu$, and so $A_t$ does not meet $C^\perp$ inside $\on{Big}(Y)$.

\begin{enumerate}[label=(\arabic*)]
    \item When the Picard rank of $Y$ is $5$ there are no $(-1)$-curves on $Y$ of the type (iv'). Hence the proof of this case is complete.
    \item When the Picard rank of $Y$ is $6$ or $7$ the only $(-1)$-curves of the type (iv') are strict transforms of conics passing through five of the blowup centres. The classes of these curves take the form
    $$2H-E_{i_1}-E_{i_2}-E_{i_3}-E_{i_4}-E_{i_5}$$
    From our previous observations it is sufficient to consider the curves
    $$2H-E_\star-E_n-E_i-E_j-E_k$$
    where $b_\star=1$. We find that $t_\diamond\geq\mu$ precisely when
    $$t_\diamond-\mu=A\cdot(C-H+E_\star)=c-a_n-a_i-a_j-a_k\geq0$$
    as required in the statement of Prop.~\ref{prop:example}.
    \item When the Picard rank of $Y$ is $8$ we obtain a further class of $(-1)$-curves from the strict transform of cubic curves in $\pr^2$ with a double point at one the blowup centres and passing through six others. The class of such a curve is
    $$3H-2E_{i_1}-E_{i_2}-\dots-E_{i_7}$$
    and again we only need consider such curves passing through $p(E_\star)$ and $p(E_n)$. There are two cases to consider: either $b_\star=1$ or $b_\star=2$. If $b_\star=2$, namely the double point of the cubic is at $p(E_\star)$, we have
    $$2(t_\diamond-\mu)=A\cdot(C-2H+2E_\star)=c-a_{i_1}-\dots-a_{i_6}$$
    which is nonnegative by assumption. This also implies the inequalities in (2) for conics. If $b_\star=1$ then
    $$t_\diamond-\mu=A\cdot(C-H+E_\star)=2c-2a_{i_1}-a_{i_2}-\dots-a_{i_6}$$
    where we choose $i_7=\star$ and this is nonnegative from the inequality in the $b_\star=2$ case and the fact that $c\geq a_{i_1}$.
    \item When the Picard rank of $Y$ is $9$ performing the same analysis with the new $(-1)$-curves coming from strict transforms of certain quartics, quintics, and sextics in $\pr^2$ again yields that $t_\diamond\geq\mu$ when the inequality in the statement of Prop.~\ref{prop:example} holds. In this case the additional inequality comes from the sextic in $\pr^2$ with a triple point at $p(E_\star)$ and a double point at $p(E_i)$ for each $i\not=\star$.
\end{enumerate}

This proves Claim \ref{claim:aux}. We have shown that $\Delta(Y,A,Y_\bullet)$ is obtained from $\Delta(Y_{n-1},\ol{A},\ol{Y}_\bullet)$ by slicing with a half-space whose boundary line passes through the point $(\mu-a_n,\beta_{\ol{Y}_\bullet}(\ol{A},\mu-a_n))$. We conclude the proof by noting that the slice we remove must correspond to removing a regular triangle of side length $a_n$ from $\Delta(Y_{n-1},\ol{A},\ol{Y}_\bullet)$ as the volume of the two polytopes differs by $a_n$ and we have already located an edge of length $a_n$.
\end{proof}

\subsection{Rational surfaces of higher Picard rank} \label{sec:higher_blowups} Next, we turn our attention to rational surfaces obtained by blowing up nine or more points in $\pr^2$.

\vspace{3pt}

Let $(Y,A)$ be a polarized rational surface presented as a tower $(\mathcal{Y},\mathcal{A})$ as in (\ref{eqn:tower}) whose only negative curves are rational $(-1)$-curves. Write $n+2$ for the Picard rank of $Y$ and let $E_\star$ be an exceptional curve on $Y$. Denote the weight sequence of $(\mathcal{Y},\mathcal{A})$ by $(c;a_1,\dots,a_{n+1})$.

\begin{lemma} \label{lem:bi_bound} Let $C$ be a rational $(-1)$-curve on the polarized surface $(Y,A)$ of the form
\[C = dH-\sum_{i=1}^{n+1}b_iE_i\]
Then if $n\geq 3$ and $d\geq 2$ we have
\[b_i\leq\frac{n-2}{n}\cdot d\quad\text{for all $i=1,\dots,n+1$.}\]
\end{lemma}

\begin{proof} After reordering we may assume that $b_{n+1}$ is the largest $b_i$. We note that
$$\sum_{i=1}^{n+1}b_i^2=d^2+1\qquad\sum_{i=1}^{n+1}b_i=3d-1$$
 Clearly $b_{n+1}\leq d-1$ since $d\geq 2$ and so write $b_{n+1}=d-k$ for some $k\geq 1$. We then have
$$\sum_{i=1}^nb_i^2=2kd-k^2+1\qquad\sum_{i=1}^nb_i=2d+k-1 \qquad\text{and thus}\qquad \sum_{i=1}^nb_i^2\leq k\sum_{i=1}^nb_i$$
By applying the Cauchy--Schwarz inequality to the latter bound, we find that
$$\left(\sum_{i=1}^nb_i\right)^2\leq\sum_{i=1}^nb_i^2\cdot\sum_{i=1}^n1=n\sum_{i=1}^nb_i^2$$
Thus we get the following bound relating $d,k$ and $n$. 
$$2d + k - 1 = \sum_{i=1}^nb_i\leq n\cdot\frac{\sum_{i=1}^n b_i^2}{\sum_{i=1}^nb_i}\leq kn \qquad\text{or equivalently}\qquad d\leq\frac{kn-k+1}{2}$$
Since $\frac{d-k}{d}$ is an increasing sequence in $d$, this gives us the following bound.
\begin{equation} \tag{$\ast$} \label{eqn:d_bound}
\frac{b_i}{d}\leq\frac{b_{n+1}}{d}=\frac{d-k}{d}\leq\frac{kn-3k+1}{kn-k+1}
\end{equation}
The upper bound in (\ref{eqn:d_bound}) is decreasing in $k$ and so the smallest value $k=1$ is optimal.
\end{proof}

\begin{prop} \label{prop:higher} Let $(Y,A)$ be the polarized rational surface above and suppose that its weight sequence satisfies
\begin{equation} \tag{$\dagger$} \label{eqn:higher_condition}
    c\geq\frac{n-2}{2}\sum_{i=1}^{n+1}a_i
\end{equation}
Let $Y_\bullet$ be the flag $\{y\in E_\star\subseteq Y\}$ for an arbitrary point $y\in E_\star$. Then the Newton--Okounkov body $\Delta(Y,A,Y_\bullet)$ has weight sequence equal to the weight sequence of $(\mathcal{Y},\mathcal{A})$.
\end{prop}

\begin{proof}
We essentially repeat the argument from Prop.~\ref{prop:example}. Note that, as for del Pezzo surfaces, the Zariski chambers for $Y$ are carved out by the hyperplanes
$$\{C^\perp\}_{C^2<0}$$
in this setting \cite[Prop.~3.4]{bks_zar_04}. We again study a blowdown $(\ol{Y},\ol{A})$ of $(Y,A)$ obtained by contracting an exceptional curve other than $E_\star$; equivalently, omitting the corresponding term from the weight sequence. Observe that if $(Y,A)$ satisfies (\ref{eqn:higher_condition}) then so does $(\ol{Y},\ol{A})$ since the Picard rank of $\ol{Y}$ is $n+1$ and the quantity $\sum_i a_i$ decreases with blowdown. Arguing recursively as before with base case a toric del Pezzo surface leaves us needing to show that every $(-1)$-curve in (iv') defines a wall that does not meet $A_t$ inside the big cone. Fix such a $(-1)$-curve $C$ whose class is notated as in Lem.~\ref{lem:bi_bound} and define $t_\diamond$ and $\mu$ as above. We compute
$$b_\star(t_\diamond-\mu)=A\cdot (C-b_\star H+b_\star E_\star)=(d-b_\star)c-\sum_{i=1}^{n+1}a_ib_i+a_\star b_\star$$
Using the bound from Lem.~\ref{lem:bi_bound} we have
$$b_\star(t_\diamond-\mu)\geq\left(d-\frac{n-2}{n}d\right)c-\sum_{i=1}^{n+1}\frac{n-2}{n}da_i=d\left(\frac{2}{n}c-\frac{n-2}{n}\sum_{i=1}^{n+1}a_i\right)$$
and the right-hand side is nonnegative when (\ref{eqn:higher_condition}) holds.
\end{proof}

\begin{remark} A weaker version of the well-known SHGH Conjecture \cite[Conj.~0.1]{def_neg_04} or \cite[Conj.~2.2.3]{chmr_var_13} states that any tower of very general blowups of $\pr^2$ produces a rational surface whose only negative curves are rational $(-1)$-curves, hence a surface to which Prop.~\ref{prop:higher} applies. We note that, like the Nagata Conjecture, the SHGH Conjecture is still very much open. \end{remark}

\subsection{Pseudo-polarized surfaces from toric degenerations} We conclude this section by discussing how degenerations can produce projective surfaces that are `more generic' with respect to their tower presentations. We use this procedure to manufacture examples of surfaces that Thm.~\ref{thm:alg_ech} and Cor.~\ref{cor:conc_emb} apply to.

\vspace{3pt}

As in \cite[\S2.4]{wor_tow_22}, weight sequences for pseudo-polarized toric surfaces are naturally indexed by the infinite rooted binary tree $\mathcal{T}$ with root vertex $\infty$, controlling the inclusion of centers to previous exceptional curves. We may view $\mathcal{T}$ as a poset with unique maximum $\infty$ and order relation generated by edges. By padding with zeroes we can view all polytopal weight sequences as elements of $\R^\mathcal{T}:=\R^{\on{vert}(\mathcal{T})}$.

\begin{definition} The space of all polytopal weight sequences $\mathcal{P}$ is given by
\[\mathcal{P} := \big\{(c,a_i) \in \R^{\mathcal{T}} \; : \; a_i \in \mathcal{T} \setminus \{\infty\} \text{ and }(c,a_i) = \on{wt}(\mathcal{Y},\mathcal{A}) \text{ for a tower }(\mathcal{Y},\mathcal{A})\big\}\]
The set $\mathcal{P}_\rho \subset \mathcal{P}$ of polytopal weight sequences in $\R^\rho$ is the subset consisting of weight sequences supported on $\rho$ vertices.
\end{definition}

\begin{prop} \label{prop:mp}
The set of polytopal weight sequences $\mathcal{P}$ is a cone in $\mathcal{R}^{\mathcal{T}}$. Moreover,
\begin{itemize}
\item $\mathcal{T}$ is concave with respect to the root component in the sense that
\[(c,a_i)\in\mathcal{P} \implies (c',a_i)\in\mathcal{P} \qquad\text{if}\qquad c'\geq c\]
\item $\mathcal{T}$ is convex with respect to all other components in the sense that
\[(c,a_i)\in\mathcal{P} \text{ and }a_i \neq 0\implies (c,a_i')\in\mathcal{P} \qquad \text{if}\qquad a_i\geq a_i'> a_j \text{ for all }j\leq i\]
\end{itemize}
\end{prop}

One can actually see that $\mathcal{P}$ is a cone directly from the framework of polarized Looijenga towers in \cite{wor_tow_22}. Here we provide a direct proof.

\begin{proof}
Suppose $(c,a_i)$ and $(c',a_i')$ are polytopal weight sequences viewed in $\mathcal{P}$. Let $\Omega$ and $\Omega'$ be polygons corresponding to these weight sequences (or rather to the associated pseudo-polarized toric surfaces). 

\vspace{3pt}

We claim that the sequence $(c+c',a_i+a_i')$ corresponds to the Minkowski sum $\Omega+\Omega'$. This is most easily seen algebro-geometrically. Let $\wt{Y}$ be a common toric blowup of $Y_\Omega$ and $Y_{\Omega'}$ and let $A,A'$ be the pullbacks of $A_\Omega$ and $A_{\Omega'}$ to $\wt{Y}$. From the definition of weight sequence we see that the pair $(\wt{Y},A+A')$ has a tower with weight sequence $(c+c',a_i+a_i')$. From standard toric geometry the polytope associated to $A+A'$ is the Minkowski sum of the polytopes associated to $A$ and $A'$, which are $\Omega$ and $\Omega'$ respectively as pullbacks of $A_\Omega$ and $A_{\Omega'}$. 

\vspace{3pt}

It is clear that scaling a weight sequence corresponds to positively scaling the corresponding polygon and so this immediately shows that $\mathcal{P}$ is a cone. Concavity in the root component follows from noting that if $A$ is the big and nef divisor corresponding to $(c,a_i)$ then $(c+d,a_i)$ corresponds to $A+p_0^*\mO_{\pr^2}(d)$, which is the sum of a big and nef divisor and a nef divisor and hence is big and nef. Convexity in the other components follows similarly from the hyperplane arrangement description of toric divisors \cite{per_div_11}.
\end{proof}

\begin{remark} Note that $\R^\rho$ is not a nice cone. It is easy to construct weight sequences supported on different sets of $\rho$ vertices of $\mathcal{T}$ whose sum is hence supported on more than $\rho$ vertices. \end{remark}

We next provide a more complete but much more implicit description of pseudo-polarized rational surfaces with polytopal weight sequences.

\vspace{3pt}



Consider the flat family $\on{Bl}_Z(\pr^2\times\pr^2)\to\pr^2$ obtained by blowing up the diagonal $Z\subseteq\pr^2$. The fibers of this family are $\pr^2$ blown up once where the center of the blowup varies across all of $
\pr^2$. Given $c,a_1\geq0$ the total space carries a line bundle $\mathscr{A}_1$ restricting to $cH-a_1E_1$ on each fiber where $H$ is the pullback of $\mO_{\pr^2}(1)$ and $E_1$ is the exceptional class. Repeating this process $n$ times yields a flat family $\mathscr{Y}$ whose fibers are all possible $n$ point blowups of $\pr^2$ (or towers of length $n$), most occurring multiple times across the family, equipped with a line bundle $\mathscr{A}$ restricting to $cH-a_1E_1-\dots-a_nE_n$ on fibers as above. A tower $(\mY,\mA)$ of length $n$ with weight sequence $(c;a_1,\dots,a_n)$ is thus exactly a fiber of this family with its restricted line bundle.

\vspace{3pt}

Suppose now that $(\mY,\mA)$ defines a pseudo-polarised rational surface $(Y,A)$ so that the restricted line bundle $A$ is big and nef. It follows from \cite[Cor.~4]{mor_cri_92} or \cite[Ex.~1.4.5]{laz_pos_17} that there is a Zariski open set of fibers containing $(\mY,\mA)$ such that the restricted line bundle on each fiber in this open set is nef. The self-intersection of each of these restricted line bundles is equal to $A^2$ by construction and so is in particular positive. By \cite[Thm.~1.4.35]{dem_van_14} it follows that each restricted line bundle is big and nef and hence each fiber in this open set defines a pseudo-polarised rational surface. In this situation we say that a tower coming from a fiber in such an open set is a \emph{central genericisation} of $(\mY,\mA)$. This name comes from the picture that one is moving the centers of the blowups defining $(\mY,\mA)$ into more general position (e.g.~from blowing up torus-invariant points on a toric surface to points not fixed by the torus action). We summarise this discussion in the following proposition.

\begin{prop} \label{prop:flat}
Suppose $(Y,A)$ is a pseudo-polarized toric surface. If $(Y',A')$ is any central genericisation of $(Y,A)$ then $A'$ is big and nef.
\end{prop}

\begin{example}
We consider the following example to illustrate the importance of $y_i'$ being in more general position than $y_i$. Let $Y$ be the del Pezzo surface of degree $6$, obtained by blowing up $\pr^2$ in three general points. This is a toric surface with $-K_Y$ ample. This surface degenerates to the surface $Y'$ obtained by blowing up $\pr^2$ in three collinear points but $-K_{Y'}$ is not even nef. In this case the general fibre $Y$ has ample anticanonical divisor but the special fibre $Y'$ has less positivity.
\end{example}

The two results of this section combine to give access to many polarized rational surfaces with polytopal weight sequences, hence whose ECH capacities can be combinatorially computed and used to effectively obstruct embeddings. Note that the nef and effective cones of such surfaces can be very complicated -- e.g.~\cite[\S1.5.D]{laz_pos_17} -- and so this is a major simplification. Prop.~\ref{prop:mp} gives a method of producing new polytopal weight sequences from old ones, and Prop.~\ref{prop:flat} implies that any pair $(Y,A)$ admitting such a weight sequence produced by a tower of blowups with sufficiently general centres is indeed (pseudo-)polarized and hence carries the structure of a symplectic $4$-manifold. We still require good Newton--Okounkov bodies in order to get the best embedding obstructions as in Cor.~\ref{cor:conc_emb} but Prop.~\ref{prop:flat} at least gives a large collection of cases where this is potentially true, and Prop.~\ref{prop:example} demonstrates this for many examples.

\subsection{Infinite staircases for non-toric surfaces}

Our final application concerns the complexity of embedding obstructions phrased in terms of `infinite staircases'. It has been seen that the obstructions from ECH capacities can be understood in terms of exceptional curves in blowups of $\pr^2$ \cite{mcd_sym_09,cri_sym_19}. Consider the ellipsoid embedding problem for a given symplectic $4$-manifold $(X,\omega)$:
$$\text{compute $f_X(z):=\inf\{r>0:E(1,z)\se (X,r\omega)\}$ for $z\geq1$}$$
It is often the case that only finitely many obstructions (coming from finitely many exceptional curves) fully describe the structure of $f_X(z)$. In all cases where infinitely many obstructions feature, $f_X(z)$ takes the form of an \emph{infinite staircase} with infinitely many points of nondifferentiability accumulating to a point where the obstruction is prescribed by the trivial volume constraint \cite{chmp_inf_20}. The existence or non-existence of staircases is a significant open problem with some intriguing conjectures such as \cite[Conj.~1.20]{chmp_inf_20} and \cite[Conj.~2.2.4]{mmw_sta_22}.

The technology of Cristofaro-Gardiner--Holm--Mandini--Pires \cite{chmp_inf_20} and Casals--Vianna \cite{cv_sha_20} to study staircases is mostly focused on the (almost) toric case. The results of \S\ref{sec:calc_via_NO_bodies} and \S\ref{sec:higher_dp}-\ref{sec:higher_blowups} allow questions about staircases for non-toric surfaces with suitable Newton--Okounkov bodies to questions about convex toric domains.

\begin{example}
To illustrate this, the existence of the Newton--Okounkov body with weight sequence $(3;1,1,1,1)$ for the del Pezzo surface of degree $5$ polarized by its anticanonical divisor described in \S\ref{sec:dp5} implies that it has an infinite staircase. This follows from the corresponding fact for any convex toric domain with weight sequence $(3;1,1,1,1)$ from \cite[Thm.~1.16]{chmp_inf_20}. This example is already known \cite[Rmk.~1.18]{chmp_inf_20} but it is a straightforward consequence of our machinery.
\end{example}

We present some more novel examples of our work to staircase questions.

\begin{prop} \label{prop:staircase}
Let $(Y,A)$ be a polarized del Pezzo surface. The ellipsoid embedding function of $(Y,\omega_A)$ has no infinite staircase when:
\begin{enumerate}
    \item $Y$ has degree $3$ and $(Y,A)$ has weight sequence
    $$(c;1,1,1,1,1,1)\qquad\text{for $4\leq c<\frac{18+\sqrt{24}}{5}\approx 4.57$}$$
    \item $Y$ has degree $1$ and $(Y,A)$ has weight sequence
    $$(c;1,1,1,1,1,1,1,1)\qquad\text{for $6\leq c<\frac{24+\sqrt{96}}{5}\approx 6.76$}$$
\end{enumerate}
\end{prop}

\begin{proof}
First, the divisors $A$ with weight sequences as in (i) and (ii) are indeed ample as the sum of the (ample) anticanonical divisor and the nef divisor $(c-3)\pi^*\mO_{\pr^2}(1)$. Observe that the lower bounds on $c$ in each case imply by Prop.~\ref{prop:example} that there is a Newton--Okounkov body for $(Y,A)$ whose weight sequence is admitted by $(Y,A)$. It follows that an infinite staircase exists for $(Y,\omega_A)$ if and only if it exists for the toric domain $X_\Delta$. In this context \cite[Thm.~1.11]{chmp_inf_20} gives that if an infinite staircase exists for $X_\Delta$ where $\on{wt}(\Delta)=(c;a_1,\dots,a_n)$ then it accumulates at $a_0>1$ where
\begin{equation} \tag{$\clubsuit$} \label{eqn:accumulation}
    a_0^2-\left(\frac{\left(3c-\sum_ia_i)^2\right)}{c^2-\sum_ia_i^2}-2\right)a_0+1=0
\end{equation}
When $c$ is bounded above as in (i) or (ii) one can verify that the discriminant of (\ref{eqn:accumulation}) is negative and hence no accumulation point can exist.
\end{proof}

Testing the discriminant is quite an unrefined way of establishing that staircases do not exist. It is very possible that more extensive results on infinite staircases for non-toric surfaces are achievable by combining our methods with other techniques from \cite{cri_spe_22,chmp_inf_20}.

\bibliographystyle{acm}

\end{document}